\documentclass[11pt,leqno]{article}
\usepackage{amsmath, amscd, amsthm, amssymb, graphics, xypic, mathrsfs}
\usepackage{fancyhdr}
\usepackage[pagebackref=true]{hyperref}
\hypersetup{backref}
\newcommand{\setof}[1]{\{ #1 \}}

\newcommand{\aone}{{\mathbb A}^1}
\newcommand{\pone}{{\mathbb P}^1}

\newcommand{\tensor}{\otimes}

\newcommand{\colim}{\operatorname{colim}}

\renewcommand{\O}{{\mathcal O}}

\newcommand{\ho}[1]{{\mathscr H}({#1})}
\newcommand{\hop}[1]{{\mathscr H}_{\bullet}({#1})}

\newcommand{\cplx}{{\mathbb C}}

\newcommand{\Z}{{\mathbb Z}}

\newcommand{\ga}{{\mathbb G}_{\bf a}}
\newcommand{\gm}{{\mathbb G}_{\bf m}}

\newcommand{\hsnis}{{\mathcal H}_s(({\mathcal Sm}_k)_{Nis})}

\newcommand{\isomto}{\stackrel{\sim}{\longrightarrow}}

\newcommand{\Spec}{\operatorname{Spec}}

\newcommand{\Sm}{{\mathcal Sm}}

\newcommand{\kbar}{{\overline{k}}}
\newcommand{\Spc}{{\mathcal Spc}}

\newcommand{\simpspc}{\Delta^{\circ}\Spc}
\newcommand{\codim}{\operatorname{codim}}
\newcommand{\Supp}{\operatorname{Supp}}
\newcommand{\real}{{\mathbb{R}}}
\newcommand{\Star}{{\operatorname{Star}}}
\newcommand{\holim}{{\operatorname{holim}}}

\theoremstyle{plain}
\newtheorem{thm}{Theorem}[section]
\newtheorem{lem}[thm]{Lemma}
\newtheorem{cor}[thm]{Corollary}
\newtheorem{prop}[thm]{Proposition}

\newtheorem{question}[thm]{Question}

\newtheorem{conj}[thm]{Conjecture}

\theoremstyle{definition}
\newtheorem{defn}[thm]{Definition}

\theoremstyle{remark}
\newtheorem{rem}[thm]{Remark}
\newtheorem{ex}[thm]{Example}
\newtheorem{entry}[thm]{}
\newtheorem{construction}[thm]{Construction}
\newtheorem{convention}[thm]{Convention}

\numberwithin{equation}{thm}
\numberwithin{figure}{subsection}

\newcommand{\AnZ}{\ensuremath{\mathbb{A}^n - Z}}

\newcommand{\shrinkmargins}[1]{
  \addtolength{\textheight}{#1\topmargin}
  \addtolength{\textheight}{#1\topmargin}
  \addtolength{\textwidth}{#1\oddsidemargin}
  \addtolength{\textwidth}{#1\evensidemargin}
  \addtolength{\topmargin}{-#1\topmargin}
  \addtolength{\oddsidemargin}{-#1\oddsidemargin}
  \addtolength{\evensidemargin}{-#1\evensidemargin}
  }

\shrinkmargins{1.0}

\begin{document}
\pagestyle{fancy}
\renewcommand{\sectionmark}[1]{\markright{\thesection\ #1}}
\fancyhead{}
\fancyhead[LO,RE]{\bfseries\footnotesize\thepage}
\fancyhead[LE]{\bfseries\footnotesize\rightmark}
\fancyhead[RO]{\bfseries\footnotesize\rightmark}
%\lhead[odd]{\scriptsize\leftmark}
%\rhead[\arabic{page}]{\footnotesize\arabic{page}}
\chead[]{}
\cfoot[]{}
\setlength{\headheight}{1cm}

\author{Aravind Asok \\ \begin{footnotesize}Department of Mathematics\end{footnotesize} \\ \begin{footnotesize}University of California, Los Angeles\end{footnotesize} \\ \begin{footnotesize}Los Angeles, CA 90095-1555 \end{footnotesize} \\ \begin{footnotesize}\url{asok@math.ucla.edu}\end{footnotesize}\\
\and Brent Doran \\ \begin{footnotesize}School of
Mathematics\end{footnotesize} \\ \begin{footnotesize}Institute for
Advanced Study\end{footnotesize} \\ \begin{footnotesize}Princeton,
NJ 08540\end{footnotesize} \\
\begin{footnotesize}\url{doranb@ias.edu}\end{footnotesize}}

\title{{\bf {${\mathbb A}^1$-homotopy groups, excision, and solvable quotients}}}
\maketitle

\begin{abstract}
We study some properties of $\aone$-homotopy groups:
geometric interpretations of connectivity, excision results, and a
re-interpretation of quotients by free actions of connected solvable groups in terms of covering spaces in the sense of $\aone$-homotopy theory.  These concepts and results are well-suited to the study of certain quotients via geometric invariant theory.

As a case study in the geometry of solvable group quotients, we investigate $\aone$-homotopy groups of smooth toric varieties. We give simple combinatorial conditions (in terms of fans) guaranteeing vanishing
of low degree $\aone$-homotopy groups of smooth (proper) toric
varieties. Finally, in certain cases, we can actually compute the
``next" non-vanishing $\aone$-homotopy group (beyond $\pi_1^{\aone}$) of a smooth toric variety.  From this point of view, $\aone$-homotopy theory, even with its exquisite sensitivity to algebro-geometric structure, is
almost ``as tractable" (in low degrees) as ordinary homotopy for
large classes of interesting varieties.
\end{abstract}

\begin{footnotesize}
\tableofcontents
\end{footnotesize}

\section{Introduction}
One of the basic problems in $\aone$-homotopy theory (see
\cite{VICM} \S 7) is to understand concretely, that is in terms of
explicit algebraic and combinatorial data, the isomorphism class of a smooth algebraic variety $X$ over a field $k$ viewed as an object in the Morel-Voevodsky $\aone$-homotopy category (see \cite{MV}).  More briefly, we want to understand concretely the $\aone$-homotopy types of smooth varieties.  Knowing the $\aone$-homotopy type of a variety $X$ provides, in a precise sense, universal cohomological information about $X$ and, following the basic idea of the theory of motives, gives insight into how geometric properties of $X$ control arithmetic properties.  

Mimicking constructions of classical homotopy theory, Morel and Voevodsky introduced analogs of homotopy groups, the so-called $\aone$-homotopy groups, into $\aone$-homotopy theory.  They proved a Whitehead theorem (see Theorem \ref{thm:Whitehead}), which shows the $\aone$-homotopy groups (strictly speaking, Nisnevich sheaves of groups) can be used to detect isomorphisms, i.e., $\aone$-weak equivalences, in the $\aone$-homotopy category.  Thus, these $\aone$-homotopy groups, while very mysterious, and extremely difficult to compute in general, provide a fundamental class of algebro-combinatorial invariants of an $\aone$-homotopy type.  The general goal of this paper is to give new techniques for computing $\aone$-homotopy groups of special smooth varieties $X$, and to thereby investigate the extent to which geometry controls arithmetic of $X$ for such varieties.

Many interesting classes of varieties, e.g., many moduli spaces, can be constructed as geometric quotients by group actions.  Typically, the space on which the group acts is comparatively easy to understand, and one is interested in studying geometry and topology of the quotient variety.  In this paper, we study $\aone$-homotopy groups of certain varieties constructed as geometric quotients.  In particular, Morel has developed a theory of covering spaces in $\aone$-homotopy theory that gives a natural geometric interpretation of the first $\aone$-homotopy group, i.e., the $\aone$-fundamental group.  A primary motivation for this paper was the realization that geometric quotients of smooth schemes by free actions of split solvable affine algebraic groups are, up to $\aone$-weak equivalence, in fact covering spaces in $\aone$-homotopy theory.  For varieties over $\cplx$, $\aone$-covering spaces {\em do not} in general give rise to covering spaces in the sense of topology, so one needs to develop somewhat different intuition for study of $\aone$-covering spaces.  We offer the following intuitive explanation: over $\real$, one sees a torsor under a split torus is simply, up to ordinary homotopy, a finite covering space in the usual topological sense.

%Geometric quotients by group actions arise naturally in algebraic geometry in moduli problems and in the construction of a host of classical varieties; typically, the space on which the group acts is comparatively easy to understand, and one is interested in the geometry of the quotient.  

%In another direction, Morel has developed a theory of $\aone$-covering spaces for smooth schemes, in analogy with classical algebraic topology. 

%A major motivation for this paper was the realization that smooth geometric quotients by free actions of (split) solvable (affine algebraic) groups are in fact $\aone$-covering spaces up to $\aone$-weak equivalence.    

In this spirit, start with a smooth scheme $X$ whose $\aone$-homotopy groups we understand.   Suppose $X$ is endowed with a free action of a split solvable affine algebraic group $G$.  If the quotient of $X$ by $G$ exists as a smooth scheme, then the $\aone$-homotopy groups of $X/G$ can be computed from those of $X$ ({\em cf.} Proposition \ref{prop:solvablequot} and Corollary \ref{cor:homotopycovering}).  In the case of additive, and more generally split unipotent, groups we exploited this to produce arbitrary dimensional moduli of $\aone$-contractible varieties \cite{AD1} as quotients of affine space.  One application of this paper is to study of $\aone$-homotopy groups of quotients of (open subsets of) affine space by free actions of split tori, which is in some sense complementary to the study of \cite{AD1}.   

%By modifying $X$ in high codimension and taking quotients by the $G$-action, we can attempt to produce other smooth schemes whose $\aone$-homotopy groups we likewise understand.  

For reductive, or even just multiplicative, group actions the
natural way to produce quotient schemes is through geometric
invariant theory, which helps identify closed subschemes that
prevent the action from being proper or free.  Much can be said
about the motivic cohomology of such quotients, by using an
inductive description of a natural stratification of the
aforementioned closed subschemes (see \cite{ADK}). Of course, when studying cohomology one constantly uses excision and localization sequences, neither of which exist as such for homotopy groups.  The technical heart of the paper is devoted to proving a version of excision with appropriate connectivity hypotheses in place (see Theorem \ref{thm:excision}).  This theorem will allow us to relate $\aone$-homotopy groups of nice smooth schemes and open subschemes with complement of high codimension.  

%After proving excision results for sufficiently ``highly connected in the $\aone$-sense" spaces (see Theorem \ref{thm:excision}), we invoke the above perspective to concretely compute some $\aone$-homotopy groups for smooth toric varieties (see Theorem \ref{thm:abelian}). 

We then invoke the perspective discussed above to concretely compute some $\aone$-homotopy groups for smooth toric varieties (see Theorem \ref{thm:abelian}).  These varieties arise as quotients of open subsets of affine space (presented as complements of arrangements of linear coordinate subspaces) by the natural (free) action of the torus dual to the Picard group of the variety.  In actuality we do not directly use GIT, which only picks out certain open subsets with a free group action (in particular, these quotients must be quasi-projective), but rather the related homogeneous coordinate ring description for toric varieties given by Cox (and many others).   This presentation captures all of the relevant open subsets in a pleasant combinatorial manner. In particular, we stress that these computations hold even for non-quasi-projective smooth toric varieties; we will discuss this more in Section \ref{ss:examples}.  We thereby see that the (fundamentally algebro-geometric) $\aone$-homotopy theory can be nearly ``as tractable" as ordinary homotopy for large classes of interesting varieties.  We now give a precise outline of the method of investigation.  

% Some comment about universal torsors?  

\subsubsection*{The topological story}
For the moment, we consider only varieties over $\cplx$ for simplicity.  Let us recall the basic techniques one can use to study the homotopy theory of quotients.  Let $G$ denote a connected, affine algebraic group and suppose $X$ is a connected smooth scheme equipped with a (left) $G$-action.  Suppose furthermore that $G$ acts freely on an open subscheme $U \subset X$, such that a (geometric) quotient $U/G$ exists as a smooth scheme.  Let us denote by $Z \subset X$ the closed complement of $U$ (equipped with the reduced induced scheme structure).  We consider the triple $(X,G,Z)$.  Assume, for simplicity of notation, that we know that $X(\cplx)$ is $(d-2)$-connected where the maximal dimensional components of $Z$ have codimension $d$ in $X$.    
\begin{itemize}
\item {\bf Excision} of a codimension $d$ closed subspace from a $(d-2)$-connected manifold induces isomorphisms on $\pi_i$ for $i \leq d-2$.  Thus $\pi_i(U(\cplx))$ vanishes for $i \leq d-2$ and $\pi_{d-1}(X(\cplx)) \cong \pi_{d-1}(U(\cplx))$.  
\item The {\bf Hurewicz} theorem gives a canonical morphism $\pi_{d-1}(U(\cplx)) \longrightarrow H_{d-1}(U(\cplx))$ that is an isomorphism if $d \geq 3$ and a surjection (indeed, abelianization) if $d = 2$.  Furthermore, {\bf Mayer-Vietoris} sequences can be used to study $H_{d-1}(U(\cplx))$ in good situations.  
\item The {\bf fibration} $U(\cplx) \longrightarrow U/G(\cplx)$ gives rise to a long exact sequence relating $\pi_i(U(\cplx))$ and $\pi_{i}(U/G(\cplx))$.    
\end{itemize}

\subsubsection*{The motivic story and its complications}
Amazingly, many of the above ideas can be adapted to study $\aone$-homotopy and $\aone$-homology.  This adaptation is by no means formal:  while notationally the results are very similar, we must prove many things to adapt the above outline, and several arithmetic subtleties rear their heads.  

\begin{itemize}
\item
{\bf Excision} holds with some additional connectivity hypotheses in place.  In particular, one must impose the extra condition that the
complement of the codimension $d$ closed subspace, here $U$, be $\aone$-connected (see Theorem \ref{thm:excision}).  Unlike with ordinary homotopy, it is {\em not} true in general that codimension at least $2$ complements of an $\aone$-connected space (see Remark \ref{rem:spacefilling}) are $\aone$-connected.    
\item
A {\bf Hurewicz} theorem relating $\aone$-homotopy and
$\aone$-homology groups has been established in work of Morel (see Definition \ref{defn:aonehomology} and
Theorem \ref{thm:aonehurewicz}).  One caveat deserves mention here.  If $\pi_1^{\aone}$ is non-trivial, then the Hurewicz morphism $\pi_1^{\aone} \longrightarrow H_1^{\aone}$ is not known to be an epimorphism (or abelianization) in general. (If $\pi_1^{\aone}$ is known independently to be abelian, the Hurewicz homomorphism is an isomorphism as expected.)  
\item The notion of a fibration in $\aone$-homotopy theory, i.e., $\aone$-fibration, is characterized by an appropriate lifting property.  While quotients as above are fibrations in classical homotopy theory, it is {\em extremely} difficult to check, and probably false in general, that they are $\aone$-fibrations.  In any case, the long exact homotopy sequence of a fibration exists in any model category in the sense of Quillen.  
\end{itemize}

In topology, the simplest examples of fibrations are covering spaces.  Analogously, the simplest $\aone$-fibrations are $\aone$-covering spaces ({\em cf.} Definition \ref{defn:geometriccovering}).  This explains our focus on quotients by solvable group actions, though other complications arise.  In stark contrast to classical covering spaces, even when $U$ is assumed $\aone$-connected, it is not {\em a priori} clear that the quotient $U/G$ is also $\aone$-connected (see e.g., Example \ref{ex:finitequotient}).  In the situations we consider, $U/G$ will be proper, and can be shown to be $\aone$-connected.  In our terminology, the map $U \longrightarrow U/G$ will be {\em a geometric Galois $\aone$-covering space} (see Section \ref{s:geometry}).

%he easiest way to guarantee that $U \longrightarrow U/G$ is an $\aone$-fibration is to ask that it be an $\aone$-covering space

The need to verify $\aone$-connectedness of $U$ (to apply the
excision theorem) and of $U/G$ (to meaningfully discuss $\aone$-homotopy groups independent of choice of base point) requires an extended foray into the geometry underlying $\aone$-connectedness.  We introduce a geometric criterion called $\aone$-chain connectedness, which for smooth proper varieties is stronger than rational chain connectivity in the sense of Koll\'ar-Miyaoka-Mori, and show how it implies $\aone$-connectedness.  

To illustrate the above method, we consider the case of smooth proper toric varieties, which can always be presented, by work of Cox (\cite{Cox}), as torus quotients.  Here $X = {\mathbb A}^n$, $G$ is a split torus, and $Z$ is a union of coordinate subspaces.  In this case, Lemma \ref{lem:affinespace} and Lemma \ref{lem:connected} establish the relevant $\aone$-connectedness hypotheses.  Consequently, the quotient morphism $U \rightarrow U/G$ is an $\aone$-covering space.  Let us emphasize that while smooth proper toric varieties are {\em simply connected} from the standpoint of classical topology (e.g., \cite{Fulton} \S 3.2 p. 56 Proposition), each such variety has {\em non-trivial} $\aone$-fundamental group!

Continuing, affine space is $\aone$-contractible (and hence highly $\aone$-connected) and the $\aone$-homology of $U$ can sometimes be computed using the Mayer-Vietoris sequence in $\aone$-homology (see Proposition \ref{prop:mayervietoris}).  One difficulty with the Mayer-Vietoris computation is that if a pair of codimension $d$ components of $Z$ intersect in codimension $d+1$ then the relevant part of the sequence does not split, so the problem is harder to analyze.  (This reflects how the constraint of {\em even arrangements}, introduced by Goresky and MacPherson in their study of the homology of arrangement complements in affine space via stratified Morse theory (see \cite{GM} Part III), is a convenient simplifying assumption.)  

The rest of the work is concerned with deducing combinatorial conditions (see Proposition \ref{prop:combinatorics}) that guarantee that the objects of study in each of the above steps
are explicitly computable.  In particular, we give sufficient
conditions for $\pi_1^{\aone}(\AnZ)$ to be abelian so we can apply the Hurewicz theorem, and for the relevant portion of the Mayer-Vietoris sequence to split (see Proposition \ref{prop:abelian}).  Conveniently, both of these have to do with whether codimension $d$
linear subspaces intersect in a variety of codimension $d+2$.

\begin{rem}
Suppose $X$ is now a variety defined over $\real$.  Our results provide corroborating evidence for the philosophy (due to Voevodsky and emphasized by Morel) that $\aone$-connectivity properties are more closely reflected by the topology of the {\em real} points.  The topological fundamental group of the real points of a smooth proper toric variety has been completely described by V. Uma (\cite{Uma}), and we refer the reader to {\em ibid.} for a longer discussion of topology of real points of toric varieties.  V. Uma also uses Cox's homogeneous coordinate ring presentation.
\end{rem}

\begin{rem}
Let us remark that if $X$ is a smooth and projective toric variety, then the Bialynicki-Birula decomposition can be used to decompose the motive of $X$.  Indeed, it follows from e.g., \cite{Brosnan} Theorem 3.3, that ${\sf M}(X)$ is a direct sum of Tate motives.  The Chow ring of a general smooth proper toric variety over a field of characteristic $0$ has been computed by Danilov (see \cite{Danilov} Theorem 10.8).  Presumably the decomposition above into Tate motives holds for all smooth proper toric varieties as well.  
\end{rem}

\begin{rem}
The algebraic K-theory of smooth, proper, equivariant compactifications of {\em non-split} tori, has been studied by Merkurjev and Panin in \cite{MerkurjevPanin}.  Proposition 5.6 of {\em ibid.} shows that a version of Cox's homogeneous coordinate ring exists for such varieties.  In other words, given such a variety $X$, one can find a torus $S$ (not necessarily split), an $S$-torsor $U$ and an $S$-equivariant embedding of $U$ into an affine space.  This $S$-torsor provides an example of a {\em universal torsor} in the sense of \cite{ColliotTheleneSansuc} 2.4.4.  In principle, many of our techniques might be extended to study $\aone$-homotopy groups of such compactifications.  
\end{rem}

\subsubsection*{Overview of sections}
We have endeavored to make this work as self-contained as possible.
Achieving this goal necessitated reviewing, at least schematically,
aspects of $\aone$-algebraic topology (as developed by Morel and Voevodsky in \cite{MV} and further by Morel in \cite{MField,MStable}) and the theory of toric varieties (see e.g., \cite{Fulton}).  While essentially none of this material is original, we hope that the novelty of presentation justifies its inclusion in the present work.

Section \ref{s:connected} is devoted to introducing an $\aone$-analog of the topological notion of path connectedness. We define a notion of $\aone$-chain connectedness (see Definition \ref{defn:connected}) and prove (see Proposition \ref{prop:aoneconnected}) that $\aone$-chain connected varieties are $\aone$-connected.  Furthermore, we discuss $\aone$-connectivity of certain open subsets of affine space.

Section \ref{s:higher} is devoted to a quick review of basic definitions and properties of higher $\aone$-homotopy and $\aone$-homology groups.  The only novelty here is that we prove the existence of Mayer-Vietoris sequences for $\aone$-homology.  Along the way, we discuss $\aone$-covering space theory and the computations of some $\aone$-homotopy groups of ${\mathbb A}^n - 0$ and ${\mathbb P}^n$, and the Postnikov tower in $\aone$-homotopy theory. 

Section \ref{s:excision} proves the first main result of the paper (see Theorem \ref{thm:excision}): an ``excision" result for $\aone$-homotopy groups.  The hardest part of this theorem is treating the case of the $\aone$-fundamental group.  Morel has developed a collection of techniques for dealing with the so-called {\em strongly $\aone$-invariant sheaves} of groups; all the $\aone$-homotopy groups of a (simplicial) space ${\mathcal X}$ are in fact strongly $\aone$-invariant, while the higher $\aone$-homotopy groups are furthermore {\em abelian}.  In \cite{CTHK}, a collection of axioms are given to prove the existence of ``Cousin resolutions" for certain ``$\aone$-invariant" sheaves of groups.  We review aspects of this theory here together with Morel's extension for strongly $\aone$-invariant sheaves of groups; these results are the key technical tools involved in proving our excision results.  

Section \ref{s:geometry} is devoted to studying the geometry of
toric varieties and, more generally, quotients by free solvable
group actions from the standpoint of $\aone$-covering space theory.
For us, a toric variety is a normal algebraic variety $X$ over a
field $k$ on which a {\em split} torus $T$ acts with an open dense
orbit.  We then review Cox's presentation of a smooth proper toric
variety as a quotient of an open subset of affine space by the
action of the dual torus of the Picard group (the latter is in fact
free abelian); we refer to this construction as the ``Cox cover."  The main results of this section are two-fold.  Proposition
\ref{prop:solvablequot}, Corollary \ref{cor:homotopycovering}, and
Proposition \ref{prop:galoiscover} together show that the Cox's
quotient presentation of any smooth proper toric variety {\em is} a
geometric Galois $\aone$-covering space (in the sense of Definition
\ref{defn:geometriccovering}).  Next, the combinatorial Proposition \ref{prop:combinatorics} sets the stage for the application of Theorem \ref{thm:excision} to study the $\aone$-homotopy groups of smooth proper toric varieties.

Section \ref{s:vanishing} is devoted to proving certain {\em vanishing} and {\em non-vanishing} results for low-degree $\aone$-homotopy groups of smooth toric varieties.  The main result of this section is Theorem \ref{thm:abelian}, which serves two purposes.  It first gives a neat combinatorial/geometric condition guaranteeing
vanishing of low degree $\aone$-homotopy groups of the Cox cover and computes the
next non-trivial (i.e., beyond $\pi_1^{\aone}$) $\aone$-homotopy
group in simple situations.  To close, this section illustrates some sample computations using these techniques.  

\begin{rem}
Let us also remark that Wendt \cite{Wendt} has provided a beautiful study of the $\aone$-fundamental group of an arbitrary smooth toric variety, using generalizations of the van Kampen theorem in $\aone$-homotopy theory.  We learned of Wendt's work after a preprint version of this work was made available.  In the language of this paper, Wendt explicitly computes the $\aone$-fundamental group of the Cox cover of {\em any} smooth proper toric variety.  Nevertheless, the $\aone$-fundamental group of a smooth proper toric variety itself is still {\em not} completely understood in general; we will discuss this point for Hirzebruch surfaces in Section \ref{s:vanishing}.  
\end{rem}

\subsubsection*{Acknowledgements}
This work was greatly helped by conversations with Fabien Morel at the Topology meeting held at the Banff International Research Center in February 2007.  He patiently explained how his structural results for $\aone$-homotopy groups might be used to prove results of ``excision" type.  We thank him kindly for these inspiring interactions and generous subsequent encouragement.  We would also like to thank Joseph Ayoub for some useful discussions at the IAS in May 2007.  Finally, we thank the referee for catching an error in the proof of the Theorem \ref{thm:excision}, and for suggesting a correction.

\subsubsection*{Conventions}
Throughout this paper $k$ will denote an arbitrary field unless
otherwise mentioned.  We will also use $\kbar$ to denote a fixed algebraic closure of $k$.  The word {\em scheme} will be synonymous with separated scheme, having essentially finite type over $k$ (i.e., a filtered limit of schemes with smooth affine transition morphisms).  A {\em variety} will be an integral scheme having finite type over $k$.  For any scheme $X$, $X^{(p)}$ will denote the set of codimension $p$ points of $X$.  If $L/k$ is a field extension, we will denote by $X_L$ the fiber product $X \times_{\Spec k} \Spec L$.  The word {\em group} will have two meanings, one in the context of schemes, and one to be explained shortly.  In the scheme theoretic context, it will mean linear algebraic group.  Following Borel, a connected {\em split solvable} group is a connected linear algebraic group admitting an increasing filtration by connected normal algebraic subgroups with sub-quotients isomorphic to $\ga$ or $\gm$.  

Let $\Sm_k$ denote the category of smooth schemes having finite type over $k$.  Unless otherwise mentioned, the word {\em sheaf} will mean Nisnevich sheaf on $\Sm_k$. Also $\Spc_k$ will denote the category of Nisnevich sheaves of sets on $\Sm_k$; thus {\em sheaf} and {\em space} will be synonymous.  The second meaning of the word group will be a group object in the category $\Spc_k$.  The Yoneda embedding $\Sm_k \hookrightarrow \Spc_k$ defined by $X \mapsto X(\cdot) = Hom_{\Sm_k}(\cdot,X)$ is a fully-faithful functor.  We use this functor to identify schemes with their corresponding spaces.  We let $\simpspc_k$ denote the category of simplicial spaces (i.e., simplicial Nisnevich sheaves on $\Sm_k$).  There is a canonical functor $\Spc_k \hookrightarrow \simpspc_k$ sending ${\mathcal X}$ to the simplicial sheaf whose $n$-th term is the space ${\mathcal X}$ and all of whose face and degeneracy maps are the identity morphism; this functor is fully-faithful.  Given a simplicial space ${\mathcal X}$, we will denote by ${\mathcal X}_n$ the sheaf of $n$-simplices of ${\mathcal X}$.  We will also use the notation $\Spc_{k,\bullet}$ (resp. $\simpspc_{k,\bullet}$) for the category of pointed (simplicial) spaces.  We will also often write $\ast$ for the one-point space (i.e., $\Spec k$), and ${\mathcal X}_+$ will denote, as in topology, the space ${\mathcal X} \coprod \ast$ pointed by $\ast$. 

We will often write ${\mathbb A}^1$ for ${\mathbb A}^1_k$ when $k$ is clear from context.  The affine line has two canonical $k$-rational points $0,1: \Spec k \longrightarrow \aone$.  Suppose ${\mathcal X}$ is a (simplicial) space.  Given a morphism $f: \aone \longrightarrow {\mathcal X}$, we will denote by $f(0)$ (resp. $f(1)$) the morphism $\Spec k \longrightarrow {\mathcal X}$
obtained by composing with the morphism $0$ (resp. $1$).

We let $\ho{k}$ and $\hop{k}$ denote the unpointed and pointed
motivic homotopy categories.  Objects of these categories are
simplicial spaces and morphisms are $\aone$-homotopy classes of
maps.  The latter will be denoted by $[{\mathcal X},{\mathcal
Y}]_{\aone}$ (resp. $[({\mathcal X},x),({\mathcal Y},y)]_{\aone}$). When discussing the $\aone$-homotopy type of a scheme $X$, we will always implicitly be considering the corresponding sheaf $X(\cdot)$ on $\Sm_k$.  Also, $\Sigma^1_s$ denotes the {\em simplicial} suspension functor.  We write ${\mathcal Ab}_k$ for the category of Nisnevich sheaves of abelian groups.

We offer a word of caution with regard to our use of the word
{\em torsor.}  If $G$ is a group scheme, a (left) $G$-torsor over a
smooth scheme $X$ consists of a triple $({\mathscr P},\pi,G)$ where
${\mathscr P}$ is a scheme equipped with a left action of $G$, $\pi:
{\mathscr P} \longrightarrow X$ is a faithfully flat, quasi-compact
$G$-equivariant morphism (for the trivial action of $G$ on $X$), and
the canonical map $G \times {\mathscr P} \longrightarrow {\mathscr
P} \times {\mathscr P}$ is an isomorphism onto ${\mathscr P}
\times_{\pi,X,\pi} {\mathscr P}$.  On the other hand, if $G$ is a
sheaf of groups, and ${\mathcal X}$ is a space, a $G$-torsor over
${\mathcal X}$ is a triple $({\mathcal P},\pi,{\mathcal X})$ where
${\mathcal P}$ is a $G$-space such that the action morphism $G
\times {\mathcal P} \longrightarrow {\mathcal P} \times {\mathcal
P}$ is a monomorphism, and such that the canonical morphism
${\mathcal P}/G \longrightarrow {\mathcal X}$ is an isomorphism.
Torsors in the former sense give rise to torsors in the latter sense
since we are assuming our schemes separated (see e.g., \cite{GIT}
Lemma 0.6).  When we speak of torsors over schemes, we will always
mean the former notion.  Finally, given a connected linear algebraic group $G$, we say that a $G$-scheme is $G$-quasi-projective if it admits an ample $G$-equivariant line bundle.

\section{$\aone$-homotopy theory and $\aone$-chain connectivity}
\label{s:connected}
The main goal of this section is to study the sheaf of $\aone$-connected components of a smooth scheme $X$.  Proceeding in na\"ive analogy with topology, one might expect that a smooth scheme $X$ is $\aone$-connected if any pair of $k$-points of $X$ lie in the image of a morphism $f: \aone \longrightarrow X$.  Statements of this form are complicated by two arithmetic issues: i) the set $X(k)$ can be empty, and ii) this kind of connectedness property need not behave well under field extensions.  Furthermore, there is an algebro-geometric distinction that occurs here:  in topology, if a pair of points can be connected by a chain of paths, then they can be connected by a single path; unfortunately, it is not clear that this is true in algebraic geometry.  Taking all of these conditions into account leads to a good geometric notion, which we call {\em $\aone$-chain connectedness} (see Definition \ref{defn:connected}).  

We begin, however, by recalling some basic notions of $\aone$-homotopy theory.  This is necessary to give a precise definition of the sheaf of $\aone$-connected components.  We also mention here that all of the results in this section that are presented without other attribution were known to Morel.  We take full responsibility for any of the (perhaps irrationally exuberant) conjectures.  

\subsubsection*{The $\aone$-homotopy category}
The $\aone$-homotopy category is constructed in the context of model categories by a two-stage categorical localization process from the category $\simpspc_k$.  Good general references for model categories are the original work of Quillen (see \cite{Quillen}) or the book by Hovey (see \cite{Hovey}).  First, one equips $\simpspc_k$ with the so-called {\em local injective} model structure (also called the Joyal-Jardine model structure).  Recall that this means that one defines cofibrations and weak equivalences in $\simpspc_k$ to be monomorphisms (of simplicial Nisnevich sheaves) and those morphisms that stalkwise induce weak equivalences of simplicial sets.  One then defines (simplicial) fibrations to be those morphisms that have the right lifting property with respect to acyclic cofibrations.  This triple of collections of morphisms equips $\simpspc_k$ with the structure of a closed model category (see e.g.,\cite{MV} \S 2 Theorem 1.4); the resulting homotopy category is denoted $\hsnis$; we will refer to this category as the {\em simplicial homotopy category}.  Morphisms between two (simplicial) spaces ${\mathcal X},{\mathcal Y}$ in $\hsnis$ will be denoted $[{\mathcal X},{\mathcal Y}]_s$ (with the field $k$ clear from context).  It is important to note that the functor $\Spc_k \longrightarrow \hsnis$, obtained from the canonical embedding $\Spc_k \hookrightarrow \simpspc_k$ is fully-faithful (see \cite{MV} \S 2 Remark 1.14).

The $\aone$-homotopy category $\ho{k}$ is constructed from by a localization in the sense of Bousfield.  One says that a (simplicial) space ${\mathcal X}$ is {\em $\aone$-local}, if for any test (simplicial) space ${\mathcal T}$, the map
\begin{equation*}
[{\mathcal T},{\mathcal X}]_s \longrightarrow [{\mathcal T} \times \aone,{\mathcal X}]_s
\end{equation*}
induced by projection is a bijection.  

\begin{defn}
\label{defn:aoneweakequivalence}
A morphism $f: {\mathcal X} \longrightarrow {\mathcal X}'$ of (simplicial) spaces is said to be an {\em $\aone$-weak equivalence} if for any $\aone$-local (simplicial) space ${\mathcal Y}$, the induced morphism
\begin{equation*}
[{\mathcal X}',{\mathcal Y}]_s \longrightarrow [{\mathcal X},{\mathcal Y}]_s
\end{equation*}
is a bijection.  
\end{defn}

\begin{defn}
\label{defn:aonefibration}
A morphism $f: {\mathcal X} \longrightarrow {\mathcal Y}$ is an {\em $\aone$-fibration} if for any morphism $j: {\mathcal A} \longrightarrow {\mathcal B}$ that is a monomorphism and $\aone$-weak equivalence, and any diagram
\begin{equation}
\label{eqn:fibrationlifting}
\xymatrix{
{\mathcal A} \ar[r]\ar[d]^j & {\mathcal X}\ar[d]^f \\
{\mathcal B} \ar[r] & {\mathcal Y},
}
\end{equation}
there exists a morphism ${\mathcal B} \longrightarrow {\mathcal X}$ making the two resulting triangles commute.  
\end{defn}

The category $\simpspc_k$ equipped with the collections of $\aone$-weak equivalences, $\aone$-fibrations and monomorphisms has the structure of a closed model category by \cite{MV} \S 2 Theorem 3.2.  We denote the associated homotopy category (obtained by localizing $\simpspc_k$ along the class of weak equivalences) by $\ho{k}$ and refer to this object as the {\em $\aone$-homotopy category}.  Let ${\mathcal H}_{s,\aone}((\Sm_k)_{Nis})$ denotes the subcategory of $\hsnis$ consisting of $\aone$-local objects.  Theorem 3.2 of {\em loc. cit.} also shows that inclusion just mentioned has a left adjoint 
\begin{equation*}
L_{\aone}: \hsnis \longrightarrow {\mathcal H}_{s,\aone}((\Sm_k)_{Nis})
\end{equation*}
that we refer to as the {\em $\aone$-localization functor.}  The category $\ho{k}$ can be identified with the category ${\mathcal H}_{s,\aone}((\Sm/k)_{Nis})$ via the $\aone$-localization functor.  We will write $[{\mathcal X},{\mathcal Y}]_{\aone}$ for the set of morphisms between ${\mathcal X}$ and ${\mathcal Y}$ in $\ho{k}$ (once more, we have suppressed $k$). 

\begin{rem}
A pointed $k$-space $({\mathcal X},x)$ is a space ${\mathcal X}$ together with a morphism of spaces $x: \Spec k \longrightarrow {\mathcal X}$.  One can make ``pointed" versions of all of these constructions by forgetting the base-point.  Thus, a morphism of pointed spaces is an $\aone$-weak equivalence if and only if the corresponding map of unpointed spaces is an $\aone$-weak equivalence.  We write $\Spc_{k,\bullet}$ for the category of pointed spaces, ${\mathcal H}_{s,\bullet}((\Sm_k)_{Nis})$ for the pointed simplicial homotopy category, and $\hop{k}$ for the pointed $\aone$-homotopy category.  
\end{rem}

Observe that in this model structure, all simplicial spaces are cofibrant. Following the general theory of model categories, in order to compute $[{\mathcal X},{\mathcal Y}]_s$ or $[{\mathcal X},{\mathcal Y}]_{\aone}$, it is necessary to choose a fibrant (or $\aone$-fibrant) replacement of ${\mathcal Y}$.  In fact, this can be done {\em functorially}.  Recall that an $\aone$-resolution functor consists of a pair $(Ex_{\aone},\theta)$, where $\theta$ is a natural transformation $Id \longrightarrow Ex_{\aone}$, and for any simplicial space ${\mathcal X}$, $Ex_{\aone}({\mathcal X})$ is simplicially fibrant, $\aone$-local (in fact this means it is $\aone$-fibrant as well by \cite{MV} \S 2 Proposition 2.28) and the morphism ${\mathcal X} \longrightarrow Ex_{\aone}({\mathcal X})$ is an $\aone$-acyclic cofibration (see \cite{MV} \S 2 Definition 3.18).

\begin{lem}
\label{lem:finproducts}
There exists an $\aone$-resolution functor that commutes with formation of finite products of simplical spaces.
\end{lem}

\begin{rem}
The construction of this functor uses the properties of the Godement resolution functor stated in \cite{MV} \S 2 Theorem 1.66 and the fact (see \cite{MV} p. 87) that by definition the functor $Sing_{*}^{\aone}(\cdot)$ commutes with formation of finite limits (and hence finite products).
\end{rem}

Consider two (simplicial) spaces ${\mathcal X},{\mathcal Y}$.  We say that two morphisms $f_0,f_1: {\mathcal X} \longrightarrow {\mathcal Y}$ are simplicial homotopy equivalent if there exists a morphism $H: {\mathcal X} \times \aone \longrightarrow {\mathcal Y}$ such that $H(0) = f_0$ and $H(1) = f_1$ (see e.g., \cite{Quillen} Ch. 2 \S 1 Definition 4).  Unfortunately, simplicial homotopy equivalence may fail to be an equivalence relation unless ${\mathcal Y}$ is $\aone$-fibrant.  If ${\mathcal Y}$ is $\aone$-fibrant, we denote this equivalence relation by $\sim_{\aone}$.  The general machinery of model categories (see e.g., \cite{Quillen} Ch. 2 \S 2 Proposition 5) then provides us with an identification
\begin{equation*}
[{\mathcal X},{\mathcal Y}]_{\aone} = Hom_{\simpspc_k}({\mathcal X},Ex_{\aone}({\mathcal Y}))/\sim_{\aone}.
\end{equation*}
Thus, we need a sufficiently explicit understanding of $Ex_{\aone}({\mathcal Y})$ to compute $\aone$-homotopy classes of maps.

\subsubsection*{The sheaf $\pi_0^{\aone}$}
Here and through the rest of this section, we fix a base field $k$.  All (simplicial) spaces, and schemes, unless otherwise noted, are defined over $k$.

Suppose ${\mathcal X}$ is a (simplicial) space.  Define $\pi_0^{\aone}({\mathcal X})$ to be the Nisnevich sheaf associated with the presheaf defined by 
\begin{equation*}
U \mapsto [U,{\mathcal X}]_{\aone}.
\end{equation*}
for $U \in {\mathcal Sm}_k$.  We will write the stalks of $\pi_0^{\aone}({\mathcal X})$ as $[S,{\mathcal X}]_{\aone}$ where $S$ is a Henselian local scheme.  Of course, this last statement requires comment as $S$ can be {\em essentially of finite type} (i.e., a(n inverse) limit of a filtering system of smooth schemes with smooth affine bonding morphisms).  Thus, we define $[S,{\mathcal X}]_{\aone}$ to be $\colim_{\alpha} [U_{\alpha},{\mathcal X}]_{\aone}$ for any choice of inverse system defining $S$; this is independent of the choice of such an inverse system.

\begin{defn}
\label{defn:aoneconnected}
A (simplicial) space  ${\mathcal X}$ is said to be {\em $\aone$-connected} if the canonical map $\pi_0^{\aone}({\mathcal X}) \to \pi_0^{\aone}(\Spec k)$ is an isomorphism, i.e., $\pi_0^{\aone}({\mathcal X})$ is the sheaf $\Spec k$.   We will sometimes refer to $\aone$-connected spaces as {\em $\aone$-$0$-connected} in analogy with classical homotopy theory.  Any space that is not $\aone$-connected will be called {\em $\aone$-disconnected}.
\end{defn}

Thus, a (simplicial) space is $\aone$-connected if and only if, for any Henselian local scheme $S$, any pair of $S$-points $s_0,s_1$ of $Ex_{\aone}({\mathcal X})$ can be connected by a morphism $S \times \aone \longrightarrow Ex_{\aone}({\mathcal X})$.  We now proceed to re-interpret this condition.  Observe that finitely generated separable extensions of $k$ are necessarily Henselian local schemes.  The following result, sometimes called the {\em unstable $\aone$-$0$-connectivity theorem}, can be deduced from the existence of the functor $Ex_{\aone}$. 

\begin{thm}[\cite{MV} \S 2 Corollary 3.22]
\label{thm:0connectivity}
Suppose ${\mathcal X}$ is a simplicial space.  The morphism of sheaves ${\mathcal X} \longrightarrow Ex_{\aone}({\mathcal X})$ induces a surjective morphism of Nisnevich sheaves
\begin{equation*}
{\mathcal X}_0 \longrightarrow \pi_0^{\aone}({\mathcal X}).
\end{equation*}
\end{thm}

Observe that an immediate consequence of this theorem is that any $\aone$-connected (simplicial) space ${\mathcal X}$ necessarily has a $k$-point.  This provides an obstruction to $\aone$-connectedness of a scheme ${\mathcal X}$.

As we observed above, to check that the sheaf $\pi_0^{\aone}({\mathcal X})$ is trivial, it suffices to check that it is trivial at all stalks.  The following result shows that $\aone$-connectedness can even be reduced to a check over fields. 

\begin{lem}[\cite{MStable} Lemma 6.1.3]
\label{lem:weaklyconnected}
The space ${\mathcal X}$ is $\aone$-connected if and only if for every separable, finitely generated extension $L/k$, the simplicial set $Ex_{\aone}({\mathcal X})(L)$ is a connected simplicial set.
\end{lem}

\subsubsection*{$\aone$-chain connectedness}
With these results in place, we can now give a geometric criterion that guarantees $\aone$-connectedness.

\begin{defn}
\label{defn:connected} 
We will say that a (simplicial) space
${\mathcal X}$ is {\em $\aone$-path connected} if for every finitely
generated, separable field extension $L/k$ the set ${\mathcal
X}_0(L)$ is non-empty and for every pair of $L$-points $x_0,x_1:
\Spec L \longrightarrow {\mathcal X}$, there exists a morphism $f:
\aone_L \longrightarrow {\mathcal X}$ such that $f(0) = x_0$ and
$f(1) = x_1$.  Similarly, we will say that a (simplicial) space
${\mathcal X}$ is {\em $\aone$-chain connected} if for every
finitely generated, separable field extension $L/k$ the set
${\mathcal X}_0(L)$ is non-empty, and for every pair of $L$-points
$x_0,x_1: \Spec L \longrightarrow {\mathcal X}$ there exist a finite
sequence $y_1,\ldots,y_{n} \in {\mathcal X}(L)$ and a collection of
morphisms $f_i: \aone \longrightarrow {\mathcal X}$ ($i = 0,
\ldots,n$) such that $f_0(0) = x_0$, $f_n(1) = x_1$, and $f_{i-1}(1)
= f_i(0) = y_i$.
\end{defn}

\begin{rem}
\label{rem:problem1}
Observe that in Definition \ref{defn:connected} it is essential that we make a statement for all separable field extensions $L/k$.  For example, given any smooth projective curve of genus $g \geq 1$ or abelian variety $X$ over a finite field $k$, $X(k)$ is a finite set.  Consider the open subscheme $U$ of $X$ obtained by removing all but a single $k$-rational point.  The trivial morphism $\aone \longrightarrow U$ whose image is the remaining $k$-point gives a chain connecting all $k$-points.  However, for such varieties, after making a finite extension $L/k$, no pair of distinct $L$-points can be connected by a morphism $\aone \longrightarrow U$ (see also Example \ref{ex:aonerigid}).  For smooth curves of genus $g \geq 2$, a similar statement can be made for $k$ a number field.
\end{rem}

\begin{prop}
\label{prop:aoneconnected}
Suppose ${\mathcal X}$ is an $\aone$-chain connected (simplicial) space over a field $k$.  Then ${\mathcal X}$ is $\aone$-connected.
\end{prop}

\begin{proof}
Suppose ${\mathcal X}$ is $\aone$-chain connected.  Consider the $\aone$-acyclic cofibration ${\mathcal X} \longrightarrow Ex_{\aone}({\mathcal X})$ given by Lemma \ref{lem:finproducts}.  By Theorem \ref{thm:0connectivity}, for any Henselian local scheme $S$ there is a canonical surjective morphism ${\mathcal X}_0(S) \longrightarrow [S,{\mathcal X}]_{\aone}$.  Since $Ex_{\aone}({\mathcal X})$ is (simplicially) fibrant and $\aone$-local, this last set can be computed in terms of simplicial homotopy classes of maps from $S$ to $Ex_{\aone}({\mathcal X})$.

Now, by Lemma \ref{lem:weaklyconnected}, to check that ${\mathcal X}$ is $\aone$-connected, it suffices to check that $[\Spec L,{\mathcal X}]_{\aone}$ is reduced to a point for every separable, finitely generated field extension $L/k$.  By the previous paragraph, we know that $[\Spec L,{\mathcal X}]_{\aone}$ is a quotient of ${\mathcal X}_0(L)$.  Since ${\mathcal X}$ is $\aone$-chain connected, it follows that ${\mathcal X}_0(L)$ is non-empty and the image of any two $L$-points of ${\mathcal X}$ in $Ex_{\aone}({\mathcal X})$ are in fact simplicially homotopy equivalent.  Thus $[\Spec L,{\mathcal X}]_{\aone}$ is reduced to a point.
\end{proof}

The definition of $\aone$-chain connectivity is, of course, more geometrically suggestive if $X$ is a smooth $k$-scheme.  Then, using full-faithfulness of the embedding $\Sm_k \hookrightarrow \Spc_k$, $\aone$-chain connectedness is a condition involving actual morphisms $
\aone \longrightarrow X$.  However, in this case, checking that $X$ is $\aone$-chain connected can be cumbersome because one has to check something for all finitely generated field extensions.  Fortunately, many of the examples considered in this paper will satisfy the hypotheses of the following easy lemma.

\begin{lem}
\label{lem:affineneighborhoods}
Suppose $X$ is an irreducible $n$-dimensional smooth scheme over a field $k$.  If $X$ admits an open cover $X = \cup_i U_i$ where each $U_i$ is ($k$-)isomorphic to ${\mathbb A}^n_k$, then $X$ is $\aone$-chain connected, and in particular $\aone$-connected.
\end{lem}

\begin{ex}
\label{ex:spherical}
In particular, as we shall see, this can be used to show that smooth proper toric varieties and flag varieties (under connected, split reductive groups) are $\aone$-connected.  More generally, the techniques used to prove these facts can be generalized to study smooth proper spherical varieties $X$ under connected, split reductive groups $G$ (here, {\em spherical} means that a Borel subgroup $B \subset G$ acts on $X$ with an open dense orbit), at least over algebraically closed fields of characteristic $0$.
\end{ex}

Thus, by definition, any $\aone$-chain-connected variety over a field $k$ admits many rational curves, and one hopes that an arbitrary $\aone$-connected smooth scheme has the same property.  

\begin{ex}
\label{ex:aonerigid}
Recall that a variety $X$ is said to be {\em $\aone$-rigid} if for any smooth scheme $U$, the canonical morphism $Hom_{\Sm_k}(U,X) \longrightarrow Hom_{\Sm_k}(U \times \aone,X)$ is bijective (see \cite{MV} \S 3 Example 2.4).  For $\aone$-rigid varieties, one has an isomorphism of sheaves $\pi_0^{\aone}(X) \cong X$; thus $\aone$-rigid varieties are {\em totally disconnected}.  Indeed, if $X$ is $\aone$-rigid, then $X$ is $\aone$-local and the result follows from the full-faithfulness of the functor $\Sm_k \longrightarrow \hsnis$.  Examples of $\aone$-rigid varieties include smooth curves of geometric genus $g \geq 1$, open subschemes of $\gm$, abelian varieties, products of such varieties, sub-schemes of such things, etc.  Observe then that, using classification results, one knows that a smooth algebraic curve is $\aone$-connected if and only if it is $\aone$-chain connected.

Furthermore, one can show that, if $k = \overline{k}$, the sheaf of $\aone$-connected components of a smooth projective ruled surface $\pi: X \longrightarrow C$ is trivial if $C$ has genus $0$ and isomorphic to the sheaf $C$ defined by the base curve of the fibration when the genus of $C$ is (strictly) greater than $0$.  In other words, the only $\aone$-connected smooth proper varieties of dimension $\leq 2$ are rational varieties.  
\end{ex}

\subsubsection*{Excision and $\aone$-connectivity}
We will not make a general comparison between $\aone$-chain connected varieties and $\aone$-connected varieties here.  Rather, we will focus on the main example that we will use for the rest of the paper: open subvarieties of affine space.

\begin{lem}
\label{lem:affinespace}
Suppose $k$ is an infinite field.  If $U \subset {\mathbb A}^n$ is an open subscheme realized as the complement of a union of coordinate subspaces (for some choice of coordinates on ${\mathbb A}^n$) of codimension $d \geq 2$, then $U$ is $\aone$-connected.  
\end{lem}

\begin{proof}
\label{cor:affinespace}
Let $k$ be an infinite field, and $K$ any finitely generated, separable extension of $k$.  Suppose $W_i$ are a sequence of linear subspaces defined over $k$ whose union has complement of codimension $\geq 2$ in ${\mathbb A}^n$.  In this case, we can even choose sequences of {\em linear} maps $\aone \longrightarrow X$ (defined over $K$) that miss the complement of the union of the $W_i$ and connect any pair of $K$-points.
\end{proof}

\begin{rem}
\label{rem:spacefilling}
Still under the assumption that $k$ is infinite, the previous result can be generalized to the following statement.  If $U \subset {\mathbb A}^n$ is an open subscheme whose complement has codimension $\geq 2$, then $U$ is $\aone$-chain connected.  The condition that $k$ be infinite is necessary.  Indeed, if $k$ is finite the complement of the codimension $2$ subscheme consisting of {\em all} $k$-rational points has no $k$-rational points and is thus $\aone$-disconnected.  Slightly more generally, we may construct ``space-filling" curves $Y \subset {\mathbb A}^n$: these are geometrically connected affine curves $Y$ over $k$ such that $Y(k) = {\mathbb A}^n(k)$ (see \cite{Katz} Lemma 1).  For such a curve, the complement $U = {\mathbb A}^n - Y$ has no $k$-rational points and thus $U$ can not be $\aone$-connected.
\end{rem}

\begin{ex}
\label{ex:finitequotient}
The following example presents a more subtle variation on the situation discussed in Remark \ref{rem:problem1}.  Observe that the diagonal action of the group $\mu_{n}$ of $n$th-roots of unity on ${\mathbb A}^m - 0$ (say $m > 1$) is scheme-theoretically free.  The quotient ${\mathbb A}^m - 0/\mu_{n}$ exists as a smooth scheme and is isomorphic to the $\gm$-torsor underlying the line bundle $\O(n)$ on ${\mathbb P}^{m-1}$.  However, the quotient morphism
$$
q: {\mathbb A}^m - 0 \longrightarrow {\mathbb A}^m - 0/\mu_{n}
$$
is {\em not} a surjective morphism at the level of $L$-points if $L$ is not algebraically closed.  Since $m > 1$, ${\mathbb A}^m - 0$ is always $\aone$-connected.  However, one can show that $\pi_0^{\aone}({\mathbb A}^m - 0/\mu_{n}) \cong \gm /\gm^{n}$, where the last sheaf is the cokernel of the $n$-th power morphism $\gm \longrightarrow \gm$ in the category of Nisnevich sheaves of abelian groups.  On the other hand, if $L$ is algebraically closed, every pair of $L$-points is in the image of a morphism $\aone \longrightarrow {\mathbb A}^m - 0 / \mu_{n}$.  We thank Fabien Morel for resolving our confusion regarding this example.
\end{ex}

\begin{conj}
\label{conj:codim2}
Let $k$ be an infinite field.  If $X$ is a smooth $\aone$-connected scheme and $U \subset X$ is an open subscheme whose complement has codimension $\geq 2$, then $U$ is $\aone$-connected.
\end{conj}

\begin{rem}
To contrast with the discussion of the previous section, let us observe that an open subset of affine space whose complement has codimension $1$ components is not $\aone$-connected.  For example, removing a hyperplane from ${\mathbb A}^n$ produces a variety $\aone$-weakly equivalent to $\gm$, which, as we discussed before, is known {\em not} to be $\aone$-connected.
\end{rem}

\section{$\aone$-homotopy and $\aone$-homology groups}
\label{s:higher}
Herein we review definitions and basic structural properties of $\aone$-homotopy sheaves of groups of degree $\geq 1$.  In particular, we give the $\aone$-version of the Whitehead Theorem (see  Theorem \ref{thm:Whitehead}), recall basics of Morel's theory of $\aone$-covers (to be discussed in greater detail in \S \ref{s:geometry}), recall the definition of $\aone$-homology (see Definition \ref{defn:aonehomology}), state Morel's $\aone$-Hurewicz theorem (see Theorem \ref{thm:aonehurewicz}), prove the existence of Mayer-Vietoris sequences (see Proposition \ref{prop:mayervietoris}), and recall Morel's computations of  low degree $\aone$-homotopy groups of ${\mathbb A}^n - 0$ and ${\mathbb P}^n$ (see Theorem \ref{thm:projectivespace}).  We also briefly discuss the Postnikov tower in $\aone$-homotopy theory, summarizing its main properties in Theorem \ref{thm:postnikov}. 

\subsubsection*{$\aone$-homotopy (sheaves of) groups}
Suppose $({\mathcal X},x)$ is a pointed (simplicial) space.  For us, the simplicial $i$-sphere is the (pointed) space $S^i_s = \Delta^{i}/\partial \Delta^i$, where $\Delta^i$ is the algebraic $n$-simplex $\Spec k[x_0,\ldots,x_i]/ (\sum_{j=0}^i x_j = 1)$.  For $i \geq 1$, the $\aone$-homotopy sheaves of groups $\pi_i^{\aone}({\mathcal X},x)$ are defined to be the sheaves associated with the presheaves
\begin{equation*}
U \mapsto [S^i_s \wedge U_+,({\mathcal X},x)]_{\aone},
\end{equation*}
which for $i \geq 2$ are sheaves of abelian groups.  Observe that these sheaves can also be interpreted as the sheaves associated with the presheaves $U \mapsto \pi_i(Ex_{\aone}({\mathcal X})(U),x)$; here we are considering homotopy groups of the corresponding (fibrant) simplicial sets.  The following result will be used without mention in the sequel.

\begin{lem}
Suppose $({\mathcal X},x)$ and $({\mathcal Y},y)$ are a pair of pointed $\aone$-connected (simplicial) spaces.  Then for any $i \geq 1$ there is a canonical isomorphism
\begin{equation*}
\pi_i^{\aone}({\mathcal X} \times {\mathcal Y}, (x,y)) \cong \pi_i^{\aone}({\mathcal X},x) \times \pi_i^{\aone}({\mathcal Y},y).
\end{equation*}
\end{lem}

\begin{proof}
This follows easily from Lemma \ref{lem:finproducts}.  Indeed, using that lemma, we reduce to the corresponding simplicial result:  the assumption that ${\mathcal X}$ be $\aone$-connected implies that $Ex_{\aone}({\mathcal X})$ is simplicially connected.
\end{proof}

\begin{rem}
\label{rem:fibration}
Formal arguments in model category theory show (see \cite{Quillen} Chapter 1 \S 3 Proposition 4) that if $f: {\mathcal X} \longrightarrow {\mathcal Y}$ is an $\aone$-fibration, then one obtains a corresponding long exact sequence of $\aone$-homotopy sheaves.
\end{rem}

\begin{defn}
\label{defn:aoneiconnected}
Suppose $k$ is an integer $\geq 0$.  A pointed (simplicial) space ${\mathcal X}$ is said to be {\em $\aone$-$k$-connected} if
\begin{equation*}
\pi_i^{\aone}({\mathcal X},x) = * \text{ for all } 0 \leq i \leq k.
\end{equation*}
Spaces that are $\aone$-$1$-connected will be called {\em $\aone$-simply connected}.
\end{defn}

Part of the reason for introducing these {\em sheaves} of $\aone$-homotopy groups, as opposed to ordinary groups, is that the $\aone$-homotopy sheaves of groups form the correct class of objects with which to detect $\aone$-weak equivalences.

\begin{thm}[$\aone$-Whitehead Theorem, \cite{MV} \S 3 Proposition
2.14]
\label{thm:Whitehead}
Suppose $f: ({\mathcal X},x) \longrightarrow ({\mathcal Y},y)$ is a morphism of pointed $\aone$-connected (simplicial) spaces.  The following conditions are equivalent:
\begin{itemize}
\item the morphism $f$ is an $\aone$-weak equivalence, and
\item for every $i > 0$, the induced morphism of $\aone$-homotopy sheaves of groups $f_*: \pi_i^{\aone}({\mathcal X},x) \longrightarrow \pi_i^{\aone}({\mathcal Y},y)$ is an isomorphism.
\end{itemize}
\end{thm}

\begin{proof}
Suppose $f$ is as in the hypotheses.  Then $f$ is an $\aone$-weak equivalence if and only if the induced morphism $Ex_{\aone}(f): Ex_{\aone}({\mathcal X}) \longrightarrow Ex_{\aone}({\mathcal Y})$ is an $\aone$-weak equivalence.  Since $Ex_{\aone}({\mathcal Y})$ is fibrant and $\aone$-local, to check that $f$ is an $\aone$-weak equivalence, one just has to check that $f$ is a weak equivalence of simplicial sheaves.  This is equivalent to the second condition by the definition of $\aone$-homotopy groups given above. 
\end{proof}

\begin{ex}
Observe that the $\aone$-homotopy groups $\pi_i^{\aone}(\Spec k)$ are all trivial if $i \geq 0$.  In particular, the $\aone$-contractible spaces of \cite{AD1} provide examples of spaces that are $\aone$-$i$-connected for all $i \geq 0$.  For this paper, it is important to observe that ${\mathbb A}^n$ is $\aone$-contractible (essentially by definition).   This fact, together with Lemma \ref{lem:affinespace} and the excision result proved in Theorem \ref{thm:excision}, will be used to give many examples of $\aone$-$i$-connected varieties for any {\em fixed} $i$.
\end{ex}

\subsubsection*{Strong $\aone$-invariance and $\aone$-homotopy sheaves}
In classical algebraic topology, one knows that the fundamental group of a (sufficiently nice) topological space is a discrete group.  The notion analogous to discreteness in $\aone$-algebraic topology is summarized in the following definition.

\begin{defn}[\cite{MField} Definition 5]
\label{defn:stronglyinvariant}
A sheaf of groups $G$ (possibly non-abelian) is {\em strongly $\aone$-invariant} if, for any $X \in {\mathcal Sm}_k$ and $i = 0,1$, the pull-back morphism $H^i_{Nis}(X,G) \longrightarrow H^i_{Nis}(X \times \aone,G)$ induced by the projection $X \times \aone \longrightarrow X$ is an isomorphism.
\end{defn}

If $G$ is a sheaf of groups, the usual simplicial bar construction gives rise to a pointed simplicial sheaf $BG$ ({\em cf.} \cite{MV} \S 4.1); we denote this pointed space by $(BG,\ast)$.  Theorem \ref{thm:0connectivity} shows that the sheaf $\pi_0^{\aone}(BG)$ is always trivial.  

\begin{entry}
\label{entry:strongaoneinvarianceequivalences}
Strong $\aone$-invariance for a sheaf of groups $G$ can be reformulated in several ways.  Proposition 1.16 of \cite{MV} \S 4 shows that for any smooth scheme $U$ one has canonical bijections $[U_+,(BG,\ast)]_s \isomto H^1_{Nis}(U,G)$, $[\Sigma^1_sU_+,(BG,\ast)]_s \isomto G(U)$, and $[\Sigma^i_s U_+,(BG,\ast)]_s$ vanishes for $i > 1$.  Using these identifications, together with the equivalences of \cite{MIntro} Lemma 3.2.1, we deduce that $G$ is strongly $\aone$-invariant if and only if $BG$ is $\aone$-local.  If $BG$ is $\aone$-local, we know that for any simplicial space ${\mathcal X}$, the canonical map $[{\mathcal X},BG]_s \longrightarrow [{\mathcal X},BG]_{\aone}$ is a bijection.  Thus, if $G$ is strongly $\aone$-invariant, we deduce the following results:
\begin{equation}
\label{eqn:bgaonelocal}
[\Sigma^i_s U_+,(BG,\ast)]_{\aone} \isomto \begin{cases} H^1_{Nis}(U,G) & \text{ if } i = 0, \\
G(U) & \text{ if } i = 1, \text{ and } \\
0 & \text{ otherwise }.\end{cases}
\end{equation}
Taken together, if $G$ is strongly $\aone$-invariant, one has $\pi_0^{\aone}(BG) = \ast$, $\pi_1^{\aone}(BG,\ast) = G$, and $\pi_i^{\aone}(BG,\ast) = 0$ for $i > 1$.  The identifications and computations above will be used repeatedly in the sequel.  
\end{entry}

\begin{ex}
It is proved in, e.g., \cite{MV} \S 4 Proposition 3.8 that $\gm$ is a strongly $\aone$-invariant sheaf of groups.  More generally, a split torus $T$ is a strongly $\aone$-invariant sheaf of groups (via an isomorphism $BT \cong B\gm^{\times n}$).  More generally, one can show that if $T$ is a torus (not necessarily split) over a perfect field that is a smooth group scheme, then $T$ is strongly $\aone$-invariant.  
\end{ex}

\begin{ex}
Observe that, on the contrary, $\ga$ is {\em not} strongly $\aone$-invariant; indeed, for any smooth scheme $X$, one has isomorphisms $H^0_{Nis}(X,\ga) \cong H^0_{Zar}(X,\ga) = H^0(X,\O_X)$.  The last group is obviously not isomorphic to $H^0_{Nis}(X \times \aone,\ga)$ for $X$ a smooth affine scheme.  Also, $GL_n$ is known {\em not} to be $\aone$-local (see, e.g., \cite{ADBundle} for a detailed discussion of this fact).  Indeed, it is well known that the canonical map $H^1_{Nis}(X,GL_n) \longrightarrow H^1_{Nis}(X \times \aone,GL_n)$ need not be an isomorphism if $X$ is not affine (e.g., there exist counter-examples with $X = \pone$).  More generally, one can show that essentially any non-abelian reductive group is {\em not} strongly $\aone$-invariant.  
\end{ex}

The following result is one of the main results of \cite{MField} and is one justification for introducing the concept of strong $\aone$-invariance.  

\begin{thm}[\cite{MField} Theorem 3.1, Corollary 3.3]
Suppose ${\mathcal X}$ is a pointed (simplicial) space.  For every $i > 0$ the sheaf of groups $\pi_i^{\aone}({\mathcal X},x)$ is strongly $\aone$-invariant.
\end{thm}

We let ${\mathcal Gr}^{\aone}_k$ denote the category of strongly $\aone$-invariant sheaves of groups.  Any homomorphism $\varphi: H \longrightarrow G$ of strongly $\aone$-invariant sheaves of groups induces a pointed map $(BH,\ast) \longrightarrow (BG,\ast)$; applying $\pi_1^{\aone}$ yields a homomorphism $H \to G$ by the discussion subsequent to Equation \ref{eqn:bgaonelocal}.  Thus one obtains a bijection
\begin{equation}
\label{eqn:nonabeliangroupstospaces}
[(BH,\ast),(BG,\ast)]_{\aone} \isomto Hom_{{\mathcal Gr}^{\aone}_k}(H,G).  
\end{equation}
This identification will be important in the proof of Theorem \ref{thm:excision}.

\subsubsection*{$\aone$-covering spaces}
Analogous to the usual theory of covering spaces, the sheaf of groups $\pi_1^{\aone}({\mathcal X},x)$ has an interpretation in terms of an $\aone$-covering space theory (see also \cite{MField} \S 4.1). In classical topology, covering spaces can be characterized by the unique path lifting property (see e.g., \cite{Spanier} Chapter 2 \S 4 Theorem 10).   One can make an analogous definition in $\aone$-homotopy theory.  

\begin{defn}
A morphism of spaces $f: {\mathcal X} \longrightarrow {\mathcal Y}$ is an {\em $\aone$-cover} if $f$ has the unique right lifting property with respect to morphisms that are simultaneously $\aone$-weak equivalences and monomorphisms.  In other words, given any square like Diagram \ref{eqn:fibrationlifting}, one requires that there exists a {\em unique} lift making both triangles commute.
\end{defn}

By definition, $\aone$-covers are $\aone$-fibrations in the sense of Definition \ref{defn:aonefibration}.  Furthermore, $\aone$-covers are closely related to strongly $\aone$-invariant sheaves of groups by the following result.  

\begin{lem}[\cite{MField} Lemma 4.5]
If $G$ is a strongly $\aone$-invariant sheaf of groups, and ${\mathcal X}$ is a space, then any $G$-torsor over ${\mathcal X}$ provides an $\aone$-cover of ${\mathcal X}$.   
\end{lem}

We will refer to an $\aone$-cover associated with a $G$-torsor under a strongly $\aone$-invariant sheaf of groups $G$ as a {\em Galois $\aone$-cover}.  

\begin{rem}
Observe that given an $\aone$-connected space ${\mathcal X}$, the total space of an $\aone$-cover need {\em not} be $\aone$-connected: take for instance the trivial $\gm$-torsor over any smooth scheme $X$.  Morel also proves that if $G$ is a finite {\em \'etale} group scheme of order coprime to the characteristic of $k$, then any $G$-torsor is an $\aone$-cover.  Once more, the total space of such a cover may be $\aone$-disconnected, e.g., a trivial torsor over a smooth scheme $X$.  More remarkably, in stark contrast to the topological situation, one can construct $\aone$-connected $\aone$-covers of $\aone$-disconnected spaces (see, e.g., Example \ref{ex:finitequotient}).
\end{rem}

If ${\mathcal X}$ is an $\aone$-connected space, we will say that $f: {\mathcal X}' \longrightarrow {\mathcal X}$ is an {\em $\aone$-covering space} if ${\mathcal X}'$ is $\aone$-connected and $f$ is an $\aone$-cover.  To emphasize, our terminology differs slightly from that of \cite{MField}: $\aone$-covering spaces are $\aone$-covers, but not conversely.  By definition, the {\em universal $\aone$-covering space} of a pointed, $\aone$-connected space $({\mathcal X},x)$ is the unique (up to unique isomorphism) pointed, $\aone$-1-connected $\aone$-covering space $(\tilde{{\mathcal X}},\tilde{x})$ of $({\mathcal X},x)$.  This space canonically has the structure of a $\pi_1^{\aone}({\mathcal X},x)$-torsor over ${\mathcal X}$.  If ${\mathcal X}$ is a pointed $\aone$-connected space, Theorem 4.8 of \cite{MField} guarantees the existence of a universal $\aone$-covering space.  

There is a Galois correspondence for $\aone$-covering spaces in analogy with the corresponding story in topology: one can construct an order reversing bijection between the lattice of strongly $\aone$-invariant normal subgroup sheaves of $\pi_1^{\aone}({\mathcal X},x)$ and pointed $\aone$-covering spaces of ${\mathcal X}$.  

%More generally, for any $\aone$-cover ${\mathcal X}' \longrightarrow {\mathcal X}$, we can consider the sheaf of groups $Aut({\mathcal X}'/{\mathcal X})$, which is always a quotient of $\pi_1^{\aone}({\mathcal X},x)$.  

\begin{defn}
Suppose $({\mathcal X},x)$ is a pointed $\aone$-connected space.  We write $Cov_{\aone}({\mathcal X})$ for the category whose objects are $\aone$-covers $({\mathcal X}',\varphi: {\mathcal X}' \longrightarrow {\mathcal X})$ of ${\mathcal X}$, and where morphisms between objects are morphisms of spaces making the obvious diagrams commute. 
\end{defn}

Given a cover $({\mathcal X}',\varphi)$, consider the fiber product diagram
\begin{equation*}
\xymatrix{
{\mathcal X}' \times_{{\mathcal X}} \ast \ar[r]^{x'}\ar[d]^{\varphi'} & {\mathcal X}' \ar[d]^{\varphi} \\
\ast \ar[r]^{x} & {\mathcal X}.
}
\end{equation*}
We know that pull-backs of $\aone$-covers are $\aone$-covers by \cite{MField} Lemma 4.7.1.  We use this repeatedly now.  Define the sheaf of groups $Aut({\mathcal X}'/{\mathcal X})$ to be the sheaf whose sections over a smooth scheme $U$ are automorphisms of the $\aone$-cover ${\mathcal X}' \times U \to {\mathcal X} \times U$.  

There is an action of the sheaf of groups $Aut({\mathcal X}'/{\mathcal X})$ by {\em deck transformations} on the space ${\mathcal X}' \times_{{\mathcal X}} \ast $.  In particular, there is a right action of $\pi_1^{\aone}({\mathcal X})$ on ${\mathcal X}' \times_{{\mathcal X}} \ast $.  Again using the fact that pull-backs of $\aone$-covers are $\aone$-covers, we observe that ${\mathcal X}' \times_{{\mathcal X}} \ast$ is an $\aone$-cover of $\ast$.  In particular, this means this fiber product is a sheaf of sets ${\mathcal S}$ that is $\aone$-invariant in the sense that for any smooth scheme $U$, the canonical map ${\mathcal S}(U) \to {\mathcal S}(\aone \times U)$ is a bijection (we have implicitly used \cite{MV} Section 2 Proposition 2.28 and \cite{MV} Section 2 Proposition 3.19).  

Let $\pi_1^{\aone}({\mathcal X},x)-{\mathcal Set}$ denote the category of sheaves of sets that are $\aone$-invariant (in the sense above) and come equipped with a right action of $\pi_1^{\aone}({\mathcal X})$.  Sending an $\aone$-cover $({\mathcal X}',\varphi)$ to the $\aone$-local space ${\mathcal X}' \times_{{\mathcal X}} \ast$ equipped with its right $\pi_1^{\aone}({\mathcal X},x)$-action thus defines a functor 
\begin{equation*}
\Gamma_x: Cov_{\aone}({\mathcal X}) \longrightarrow \pi_1^{\aone}({\mathcal X},x)-{\mathcal Set},
\end{equation*}
which we now study in the Galois theoretic terms just introduced.  

As above, let $\tilde{{\mathcal X}}$ denote the $\aone$-universal cover of ${\mathcal X}$ with its prescribed $\pi_1^{\aone}({\mathcal X},x)$-action.  Given an object in ${\mathcal S} \in \pi_1^{\aone}({\mathcal X},x)-{\mathcal Set}$, consider the contracted product space $\tilde{{\mathcal X}} \times^{\pi_1^{\aone}({\mathcal X},x)} {\mathcal S}$, which is the quotient of $\tilde{{\mathcal X}} \times {\mathcal S}$ by the obvious actions of $\pi_1^{\aone}({\mathcal X},x)$.    

\begin{prop}
\label{prop:contractedproduct}
Projection onto the first factor determines a morphism $\tilde{{\mathcal X}} \times^{\pi_1^{\aone}({\mathcal X},x)} {\mathcal S} \longrightarrow {\mathcal X}$ that is an $\aone$-cover.  Moreover, the aforementioned contracted product determines a functor 
\begin{equation*}
\pi_1^{\aone}({\mathcal X},x)-{\mathcal Set} \longrightarrow Cov_{\aone}({\mathcal X}). 
\end{equation*}
\end{prop}

\begin{proof}
We will reduce the check that $\tilde{X} \times^{\pi_1^{\aone}({\mathcal X},x)} {\mathcal S} \to {\mathcal X}$ is an $\aone$-cover to a ``classical" simplicial fact whose proof is straightforward but tedious.  We follow the lines of the proof of Theorem 4.8 of \cite{MField}.  Let $Ex_{\aone}({\mathcal X})$ be a fibrant and $\aone$-local replacement for ${\mathcal X}$.  Let $\tilde{{\mathcal X}}_{\aone}$ denote the universal cover of $Ex_{\aone}({\mathcal X})$ in the {\em simplicial} sense introduced on p. 116 of \cite{MField}, i.e., $\tilde{{\mathcal X}}_{\aone} \to Ex_{\aone}({\mathcal X})$ is a simplicial covering, which means it has unique right lifting with respect to cofibrations that are simplicial weak equivalences) and it is simplicially $1$-connected.  Consider the space $\tilde{{\mathcal X}}_{\aone}\times^{\pi_1^{\aone}({\mathcal X},x)} {\mathcal S}$, defined as above, which fits into a diagram of the form
\[
\tilde{{\mathcal X}}_{\aone} \times {\mathcal S}\to\tilde{{\mathcal X}}_{\aone}\times^{\pi_1^{\aone}({\mathcal X},x)} {\mathcal S} \to Ex_{\aone}({\mathcal X}).
\]
Since ${\mathcal S}$ is fibrant and $\aone$-local, it suffices to show that the last morphism is in fact a {\em simplicial} covering, since in that case it must be an $\aone$-covering.  This fact is proven along the same lines as the ``classical" argument using an open cover over which the universal covering morphism trivializes, though one now uses \u Cech simplicial schemes and the fact that ${\mathcal S}$ is fibrant.  

In any case, using the above discussion, pulling back the last sequence of morphisms along ${\mathcal X} \to Ex_{\aone}({\mathcal X})$, and using right properness of the $\aone$-model structure (i.e., that pull-backs of $\aone$-weak equivalences along $\aone$-fibrations are $\aone$-weak equivalences) we then conclude that ${\mathcal X} \times^{\pi_1^{\aone}({\mathcal X},x)} {\mathcal S}$ is in fact an $\aone$-cover of ${\mathcal X}$. 
\end{proof}

The next result follows from the construction of the functors above.  

\begin{thm}[{\em cf.} \cite{MField} Remark 4.10]
\label{thm:aonecoveringspacedictionary}
The functor $\Gamma_x: Cov_{\aone}({\mathcal X}) \longrightarrow \pi_1^{\aone}({\mathcal X},x)-{\mathcal Set}$ and the contracted product functor of Proposition \ref{prop:contractedproduct} induce mutually inverse equivalences of categories.   
\end{thm}

Recall the discussion of Paragraph \ref{entry:strongaoneinvarianceequivalences}.  Suppose $({\mathcal X},x)$ is a pointed space and $G$ is a strongly $\aone$-invariant sheaf of groups.  The set $[{\mathcal X},BG]_{\aone}$ of un-pointed homotopy classes of maps is in canonical bijection with the set of $G$-torsors on ${\mathcal X}$ by \cite{MV} \S 4 Proposition 1.15.  Thus, forgetting base-points gives rise to a map
\begin{equation}
\label{eqn:pointedtounpointed}
[({\mathcal X},x),(BG,\ast)]_{\aone} \longrightarrow [{\mathcal X},BG]_{\aone}.
\end{equation}
We would like to give a geometric description of the set $[({\mathcal X},x),(BG,\ast)]_{\aone}$.

\begin{lem}
Map \ref{eqn:pointedtounpointed} is surjective and identifies the set on right-hand side with the quotient of $[({\mathcal X},x),(BG,\ast)]_{\aone}$ by the natural conjugation action of $\pi_1^{\aone}(BG)(k) = G(k)$.
\end{lem}

\begin{proof}
Since $BG$ is already $\aone$-local, taking a fibrant replacement we may assume $BG$ is fibrant and $\aone$-local.  This reduces us to considering the corresponding statements for simplicial sheaves.  By evaluation on stalks, these last statements immediately reduce to the corresponding statements for simplicial sets.  
\end{proof}

\begin{cor}
\label{cor:torsordescription}
The set $[({\mathcal X},x),(BG,\ast)]_{\aone}$ is in canonical bijection with the set of isomorphism classes of pairs consisting of a $G$-torsor on ${\mathcal X}$ and an element $g \in G(k)$, i.e., a trivialization of the ($\aone$-invariant) fiber over $x$ of the pull-back to ${\mathcal X}$ of the universal $G$-torsor over $BG$.  
\end{cor}

\begin{proof}
Consider the function that sends a pointed morphism ${\mathcal X} \longrightarrow BG$ to the underlying $G$-torsor ${\mathcal P}$ and the fiber $\Gamma_x({\mathcal P})$.  Since ${\mathcal P}$ is a $G$-torsor, it follows that $\Gamma_x({\mathcal P})$ is non-canonically isomorphic to the sheaf of groups $G$.  The composite map $\ast \longrightarrow {\mathcal X} \longrightarrow BG$ gives rise to a trivial $G$-torsor over $x$ and thus to an isomorphism $G \isomto \Gamma_x({\mathcal P})$.  Now, this isomorphism of sheaves with right $G$-action is uniquely determined by the image of $1 \in G$, i.e., a homomorphism $\ast \longrightarrow G$.  Also, $Hom_{\Spc_k}(\ast,G) := Hom_{\Spc_k}(\Spec k,G)$ and this last set is by definition $G(k)$.  We leave the reader the task of writing down the inverse map.  
\end{proof}

\subsubsection*{Geometric $\aone$-coverings}
\begin{defn}
\label{defn:geometriccovering}
A {\em geometric $\aone$-covering} of an $Y \in {\mathcal Sm}_k$ is a morphism of schemes $f: X \longrightarrow Y$ that makes $X$ into an $\aone$-covering space of $Y$.  If furthermore, $f: X \longrightarrow Y$ is a torsor under a strongly $\aone$-invariant sheaf of groups, then $f$ will be called a {\em geometric Galois $\aone$-covering.}\end{defn}

\begin{rem}
In \S \ref{s:geometry}, we will study geometric $\aone$-covering spaces of $\aone$-connected smooth schemes coming from torsors under split tori.  Not all $\aone$-covering spaces need be geometric.  The main complicating feature of this discussion is that, as we noted above, it is non-trivial to check that the total space of an $\aone$-cover of an $\aone$-connected smooth scheme defined by a $G$-torsor, for $G$ a split torus or a finite \'etale group scheme of order coprime to the characteristic of $k$, is itself $\aone$-connected.  On the other hand, Morel states  (see \cite{MICM} Remark 3.9) that an $\aone$-connected smooth scheme admits no finite \'etale covers of order coprime to the characteristic of the base field.

Also, the following might help to explain why $\gm$-torsors ``ought to be" $\aone$-covers.  Suppose $X$ is a smooth scheme defined over a field $k$ which is embeddable in $\real$.  Suppose $L \longrightarrow X$ is a $\gm$-torsor.  Observe that $L(\real) \longrightarrow X(\real)$ is then homotopy equivalent to a covering space of $X(\real)$ with group $\Z/2\Z$.
\end{rem}

\subsubsection*{$\aone$-homology}
Analogous to the $\aone$-homotopy (sheaves of) groups, one can define $\aone$-homology (sheaves of) groups and reduced $\aone$-homology (sheaves of) groups.  Henceforth, we will suppress the modifier ``sheaves of" and refer just to ``groups."  To define the aforementioned objects, one uses the $\aone$-derived category as constructed by Morel in \cite{MField} \S 3.2.  This construction proceeds along the same lines as the construction of the $\aone$-homotopy category.  One first considers the category of (unbounded) chain complexes of sheaves of abelian groups on $Sm_k$, which we denote by $C_*({\mathcal Ab}_k)$, then one equips it with an appropriate model category structure via an appropriate notion of $\aone$-weak equivalence.   

Let $D({\mathcal Ab}_k)$ denote the usual (unbounded) derived category of chain complexes of sheaves of abelian groups.  Let $\Z(X)$ denote the free sheaf of abelian groups on $X$.  One can define a chain complex $C_*$ to be {\em $\aone$-local} if for any chain complex $D_*$ the projection $D_* \tensor \Z(\aone) \longrightarrow D_*$ induces a bijection
\begin{equation*}
Hom_{D({\mathcal Ab}_k)}(D_*,C_*) \longrightarrow Hom_{D({\mathcal Ab}_k)}(D_* \tensor \Z(\aone),C_*).
\end{equation*}
A morphism $f: C_* \longrightarrow D_*$ is an {\em $\aone$-quasi-isomorphism} if for any $\aone$-local chain complex $E_*$ the induced morphism
\begin{equation*}
Hom_{D({\mathcal Ab}_k)}(D_*,E_*) \longrightarrow Hom_{D({\mathcal Ab}_k)}(C_*,E_*)
\end{equation*}
is bijective. Define cofibrations to be monomorphisms, weak equivalences to be $\aone$-quasiisomorphisms, and $\aone$-fibrations to be those morphisms having the right lifting property with respect to morphisms that are simultaneously monomorphisms and $\aone$-quasiisomorphisms.  We let $D_{\aone-loc}({\mathcal Ab}_k)$ denote the full subcategory of $\aone$-local objects.  Lemma 3.16 of \cite{MField} shows that the inclusion functor admits a left adjoint
\begin{equation*}
L_{\aone}: D({\mathcal Ab}_k) \longrightarrow D_{\aone-loc}({\mathcal Ab}_k)
\end{equation*}
called the $\aone$-localization functor.  The homotopy category for the above model structure will be denoted $D_{\aone}(k)$ and called the {\em $\aone$-derived category}.  The $\aone$-localization functor allows one to identify $D_{\aone}(k)$ with $D_{\aone-loc}({\mathcal Ab}_k)$.  The next result follows from the construction of the functor $L_{\aone}$ in the same way manner as \cite{MStable} Corollary 4.2.3.2.  

\begin{lem}
\label{lem:localizationexact}
The functor $L_{\aone}$ preserves exact triangles.  
\end{lem}

The (sheaf-theoretic) Dold-Kan correspondence (see \cite{MV} \S 1.2) gives an adjoint equivalence between the category of simplicial abelian groups and the category of chain complexes (differential of degree $-1$) of abelian groups.  In one direction, this construction sends a complex $A_*$ to the corresponding Eilenberg-MacLane space $K(A_*)$.  In the other direction, we let $C_*(\Z({\mathcal X}))$ denote the normalized chain complex associated with the free simplicial sheaf of abelian groups on ${\mathcal X}$.

\begin{defn}
\label{defn:singularchaincomplex}
The {\em $\aone$-singular chain complex} of ${\mathcal X}$, denoted $C_*^{\aone}({\mathcal X})$, is the $\aone$-localization $L_{\aone}(C_*(\Z({\mathcal X})))$.
\end{defn}

If $({\mathcal X},x)$ is a pointed space, the map $\Z \longrightarrow C_*^{\aone}({\mathcal X})$ induced by the base-point splits the morphism $C_*^{\aone}({\mathcal X}) \longrightarrow \Z$ induced by the structure morphism ${\mathcal X} \longrightarrow \Spec k$.  The kernel of the morphism $C_*^{\aone}({\mathcal X}) \longrightarrow \Z$ is called the {\em reduced $\aone$-singular chain complex} of $({\mathcal X},x)$ and is denoted $\tilde{C}_*^{\aone}({\mathcal X},x)$; there is a direct sum decomposition $C_*^{\aone}({\mathcal X}) \cong \Z \oplus \tilde{C}_*^{\aone}({\mathcal X},x)$.

\begin{defn}
\label{defn:aonehomology}
If ${\mathcal X}$ is a (simplicial) space, the {\em $\aone$-homology groups} $H_n^{\aone}({\mathcal X})$ are defined to be the homology sheaves $H_n(C_*^{\aone}({\mathcal X}))$.  If $({\mathcal X},x)$ is a pointed (simplicial) space, the {\em reduced $\aone$-homology groups} $\tilde{H}_n^{\aone}({\mathcal X},x)$ are defined to be the homology sheaves $H_n(\tilde{C}_*^{\aone}({\mathcal X},x))$.
\end{defn}

One could also define $\aone$-cohomology $H^n_{\aone}({\mathcal X})$ of a space ${\mathcal X}$ by taking cohomology of the (co-chain) complex $Hom(C_*^{\aone}({\mathcal X}),\Z)$.  It is clear from the definitions that reduced $\aone$-homology commutes with simplicial suspension in the sense that we have canonical isomorphisms
\begin{equation*}
\tilde{H}^{\aone}_i(\Sigma^1_s ({\mathcal X},x)) \cong \tilde{H}^{\aone}_{i+1}({\mathcal X},x).
\end{equation*}

\begin{rem}
Given a $X \in {\mathcal Sm}_k$, Morel has shown that $H_i^{\aone}(X)$ vanishes for $i < 0$ (see \cite{MField} Corollary 3.31).  Furthermore, he has conjectured that for any smooth scheme $X$ of dimension $n$, $H_i^{\aone}(X)$ vanishes for $i > 2n$.  If $X$ is a smooth affine scheme of dimension $n$, Morel has also conjectured that $H_i^{\aone}(X)$ vanishes for $i > n$.
\end{rem}

\begin{rem}
By \cite{MVW} Definition 10.8, the Nisnevich homology sheaves of the motive ${\sf M}(X)$ give rise to the Suslin algebraic singular homology sheaves of $X$, which we denote by $H_i^{sus}(X)$.  One can construct a canonical morphism $H_i^{\aone}(X) \longrightarrow H_i^{sus}(X)$, which is not an isomorphism in general ({\em cf.} \cite{MICM} Remark 3.12).  
\end{rem}

\subsubsection*{Strict $\aone$-invariance}
The $\aone$-homology groups $H_i^{\aone}({\mathcal X})$ are also ``discrete" from the standpoint of $\aone$-homotopy theory.  They have a structure that is, {\em a priori}, stronger than strong $\aone$-invariance; the following definition is due to Morel (and historically preceded by work of Voevodsky and Rost).  

\begin{defn}[\cite{MField} Definition 5]
\label{defn:strictly}
A {\em strictly $\aone$-invariant sheaf of groups} is a sheaf of groups $A$ such that for any $X \in {\mathcal Sm}_k$, and every $i \geq 0$, the pull-back map $H^i_{Nis}(X,A) \longrightarrow H^i_{Nis}(X \times \aone,A)$ induced by projection is a bijection.
\end{defn}

\begin{entry}
\label{entry:strictaoneinvarianceequivalences}
We can discuss strictly $\aone$-invariant sheaves of groups along the same lines as in Paragraph \ref{entry:strongaoneinvarianceequivalences}.  To do this, recall that using the sheaf theoretic Dold-Kan correspondence, one can consider for any sheaf of abelian groups $A$ and positive integer $i$, the Eilenberg-MacLane space $K(A,i)$. For any smooth scheme $U$, Proposition 1.26 of \cite{MV} \S 2 then gives the identification $[U,K(A,i)]_s \isomto H^i_{Nis}(U,A)$.  Since $A$ is abelian, there is a corresponding statement for base-pointed maps as well: $[U_+,(K(A,i),\ast)]_s \isomto H^i_{Nis}(U,A)$.  Using these identifications together with Lemma 3.2.1 of \cite{MIntro}, we deduce that $A$ is strictly $\aone$-invariant if and only if $K(A,i)$ is $\aone$-local for each $i \geq 0$.  If $K(A,i)$ is $\aone$-local, we deduce that for any pointed simplicial space $({\mathcal X},x)$ the canonical map $[({\mathcal X},x),(K(A,i),\ast)]_s \longrightarrow [({\mathcal X},x),(K(A,i),\ast)]_{\aone}$ is a bijection.  Thus, if $A$ is strictly $\aone$-invariant, we conclude that we get an isomorphism of groups:
\begin{equation}
[U_+,(K(A,i),\ast)]_{\aone} \isomto H^i_{Nis}(U,A).
\end{equation}
Combining the discussion of \cite{MV} p.58-59 (see also \cite{MStable} p. 23) with the above shows that if $A$ is a strictly $\aone$-invariant sheaf of groups, the $\aone$-homotopy sheaf $\pi_j^{\aone}(K(A,i),\ast)$ vanishes if $j \neq i$ and is equal to $A$ if $i = j$.  
\end{entry}

\begin{thm}[\cite{MField} Proposition 3.22 and Theorem 3.25]
If $A$ is a strongly $\aone$-invariant sheaf of {\em abelian} groups, then $A$ is in fact strictly $\aone$-invariant .  If $({\mathcal X},x)$ is a pointed (simplicial) space, $\pi_i^{\aone}({\mathcal X},x)$ is a strictly $\aone$-invariant sheaf of groups for $i \geq 2$, and $H_i^{\aone}({\mathcal X})$ is strictly $\aone$-invariant for any $i \geq 0$.    
\end{thm}

\begin{rem}
It was proved earlier (see \cite{MStable} Remark 8 and Theorem 6.2.7) that the $\aone$-homology sheaves $H_i^{\aone}({\mathcal X})$ are strictly $\aone$-invariant.
\end{rem}

Let ${\mathcal Ab}^{\aone}_k$ denote the category of strictly $\aone$-invariant sheaves of groups.  Morel has shown (see \cite{MStable} Lemma 6.2.13) that ${\mathcal Ab}^{\aone}_k$ is in fact abelian, though we will not need this fact.  Note that if $\varphi: A \longrightarrow A'$ is a morphism of strictly $\aone$-invariant sheaves of groups, applying the functor $K(\cdot,i)$ induces a map $K(A,i) \longrightarrow K(A',i)$;  applying the functor $\pi_i^{\aone}(\cdot)$ produces a homomorphism $A \to A'$.  Thus, one obtains a bijection
\begin{equation}
\label{eqn:abeliangroupstospaces}
[K(A,i),K(A',i)]_{\aone} \isomto Hom_{{\mathcal Ab}^{\aone}_k}(A,A').  
\end{equation}
This bijection will be important in the proof of Theorem \ref{thm:excision}

\subsubsection*{Postnikov Towers}
Our present goal is to identify ``cohomology computations" with homomorphisms in ${\mathcal Gr}^{\aone}_k$ and ${\mathcal Ab}^{\aone}_k$.  To do this, we use the Postnikov tower, which we now quickly recall.  The main references for this section are \cite{MV} p. 56 and \cite{MStable} \S 3.2 (in the stable case).  

Suppose ${\mathcal X}$ is a (simplicial) space.  For simplicity, we only consider pointed, connected (simplicial) spaces $({\mathcal X},x)$, and by making an $\aone$-fibrant replacement of ${\mathcal X}$, we can assume ${\mathcal X}$ is $\aone$-fibrant as well.  Recall that the $m$-th level of the Postnikov tower $P^{(m)}({\mathcal X})$ of the space ${\mathcal X}$ is the sheaf associated with the presheaf
\begin{equation*}
U \mapsto Im({\mathcal X}(U) \rightarrow cosk_m{\mathcal X}(U)).
\end{equation*}
By construction, there are morphisms ${\mathcal X} \longrightarrow P^{(m)}({\mathcal X})$ and stalkwise fibrations (though not, in general, fibrations for the injective model structure) $p_{m+1}: P^{(m+1)}({\mathcal X}) \longrightarrow P^{(m)}({\mathcal X})$; these morphisms fit into an obvious commutative triangle.  

There is a canonical morphism ${\mathcal X} \longrightarrow \holim_{m} P^{(m)}{\mathcal X}$.  Since the site $(\Sm_k)_{Nis}$ is a site of finite type (by \cite{MV} \S 2 Theorem 1.37 coupled with \cite{MV} \S 3 Proposition 1.8), {\em ibid.} \S 2 Definition 1.31 says that this morphism is a simplicial weak equivalence.  By construction, the Postnikov tower is covariantly functorial in ${\mathcal X}$.  Furthermore, {\em ibid.} \S 2 Proposition 1.36, the definition of the $\aone$-homotopy groups, and the long exact sequence in homotopy groups of a fibration shows that the homotopy fiber of $p_m$ is $\aone$-weakly equivalent to the Eilenberg-MacLane space $K(\pi_m^{\aone}({\mathcal X}),m)$.  

Summarizing, the spaces $P^{(m)}({\mathcal X})$ have the property that the maps $\pi_i^{\aone}({\mathcal X},x) \longrightarrow \pi_i^{\aone}(P^{(m)}({\mathcal X}))$ are isomorphisms for $i \leq m$, the homotopy sheaves of groups $\pi_i^{\aone}(P^{(m)}({\mathcal X}))$ vanish for $i > m$, and the homotopy fiber of the map $p_{m}$ is a $K(\pi_m^{\aone}({\mathcal X},x),m)$.  In the special case where ${\mathcal X}$ is $\aone$-$(m-1)$-connected, $P^{(m)}({\mathcal X})$ is $\aone$-weakly equivalent to the Eilenberg-MacLane space $K(\pi_m^{\aone}({\mathcal X}),m)$.  Thus, we obtain a canonical morphism
\begin{equation*}
{\mathcal X} \longrightarrow K(\pi_m^{\aone}({\mathcal X}),m)
\end{equation*}
that is an isomorphism on $\aone$-homotopy groups of degree $\leq m$. 

We now use the discussion of Paragraphs \ref{entry:strongaoneinvarianceequivalences} and \ref{entry:strictaoneinvarianceequivalences}, together with the identifications of Equations \ref{eqn:nonabeliangroupstospaces} and \ref{eqn:abeliangroupstospaces}.  In the first case, if $({\mathcal X},x)$ is a pointed, $\aone$-connected space, and $G$ is a strongly $\aone$-invariant sheaf of groups, then the map ${\mathcal X} \longrightarrow P^{(1)}({\mathcal X})$ induces a canonically defined, functorial map
\begin{equation}
\label{eqn:nonabeliancohomology}
[({\mathcal X},x),(BG,\ast)]_{\aone} \longrightarrow Hom_{{\mathcal Gr}^{\aone}_k}(\pi_1^{\aone}({\mathcal X},x),G)
\end{equation}
together with an explicit map in the reverse direction.  Similarly, if $A$ is a strictly $\aone$-invariant sheaf of groups, and $({\mathcal X},x)$ is pointed and $\aone$-$(m-1)$-connected for some integer $m \geq 2$, then the map ${\mathcal X} \longrightarrow P^{(m)}({\mathcal X})$ induces a canonically defined, functorial map
\begin{equation}
\label{eqn:abeliancohomology}
H^m_{Nis}({\mathcal X},A) \cong [({\mathcal X},x),K(A,m)]_{\aone} \longrightarrow Hom_{{\mathcal Ab}^{\aone}_k}(\pi_m^{\aone}({\mathcal X},x),A)
\end{equation}
together with an explicit map in the reverse direction.  We summarize our discussion with the following result (see, e.g., \cite{MField} Remark 4.11 or \cite{MBundle} Lemma B.2.2).

\begin{thm}
\label{thm:postnikov}
Let $({\mathcal X},x)$ be a pointed $\aone$-connected space.  If $G$ is any strongly $\aone$-invariant sheaf of groups, we have a functorial  bijection
\begin{equation*}
[({\mathcal X},x),(BG,\ast)]_{\aone} \isomto Hom_{{\mathcal Gr}^{\aone}_k}(\pi_1^{\aone}({\mathcal X},x),G).
\end{equation*}
Suppose $m$ is an integer $\geq 2$.  If furthermore ${\mathcal X}$ is $\aone$-$(m-1)$-connected space, and $A$ is a strictly $\aone$-invariant sheaf of abelian groups, then there is a functorial bijection
\begin{equation*}
H^{m}_{Nis}({\mathcal X},A) \isomto Hom_{{\mathcal Gr}^{\aone}_k}(\pi_{m}^{\aone}({\mathcal X},x),A).
\end{equation*}
\end{thm}

\begin{proof}[Sketch of proof.]
In each case, surjectivity is clear by the discussion preceding the statement.  Thus, it suffices to prove injectivity.  In the case where ${\mathcal X}$ is pointed and connected, the injectivity statement in the case of strongly $\aone$-invariant sheaves of groups is contained in \cite{MBundle} B.2.2 p. 59.  A similar method works to prove injectivity in the remaining cases.  In all cases, choosing explicit representing morphisms of homotopy classes, one uses functoriality of the Postnikov tower to reduce to the identifications of Equations \ref{eqn:nonabeliangroupstospaces} and \ref{eqn:abeliangroupstospaces}.
\end{proof}

\subsubsection*{Some computational tools}
Sending a space to its (reduced) $\aone$-singular chain complex gives a functor $\ho{k} \longrightarrow D_{\aone}({\mathcal Ab}_k)$ (resp. $\hop{k} \longrightarrow D_{\aone}({\mathcal Ab}_k)$).  Via adjunction and the Dold-Kan correspondence, for any pointed space $({\mathcal X},x)$ there is an induced $\aone$-Hurewicz morphism $\pi_i^{\aone}({\mathcal X},x) \longrightarrow H_i^{\aone}({\mathcal X},x)$.  The structure of these Hurewicz morphisms is summarized in the following result.  

\begin{thm}[$\aone$-Hurewicz Theorem (\cite{MField} Theorems 3.35 and 3.57)]
\label{thm:aonehurewicz}
Let $n \geq 1$ be an integer and let $({\mathcal X},x)$ be a pointed  $\aone$-connected (simplicial) space.
\begin{itemize}
\item[i)] The $\aone$-Hurewicz morphism $\pi_1^{\aone}({\mathcal X},x) \longrightarrow H^{\aone}_1({\mathcal X})$ is the initial morphism from $\pi_1^{\aone}({\mathcal X},x)$ to a strongly $\aone$-invariant sheaf of abelian groups.
\item[ii)] If $\pi_1^{\aone}({\mathcal X},x)$ is abelian, the morphism of the previous statement is an isomorphism.
\item[iii)] If $n > 1$, and ${\mathcal X}$ is $\aone$-$(n-1)$-connected, then $H_i^{\aone}({\mathcal X})$ vanishes if $0 \leq i \leq n-1$, the $\aone$-Hurewicz morphism $\pi_n^{\aone}({\mathcal X},x) \longrightarrow H_n^{\aone}({\mathcal X})$ is an isomorphism, and $\pi_{n+1}^{\aone}({\mathcal X},x) \longrightarrow H_{n+1}^{\aone}({\mathcal X})$ is an epimorphism.
\end{itemize}
\end{thm}

It is expected though not known that, in general, the $\aone$-Hurewicz morphism $\pi_1^{\aone}({\mathcal X},x) \longrightarrow H^{\aone}_1({\mathcal X})$ is an epimorphism and identifies $H^{\aone}_1({\mathcal X})$ as the abelianization of $\pi_1^{\aone}({\mathcal X},x)$.

\begin{prop}[Mayer-Vietoris]
\label{prop:mayervietoris}
Suppose $X \in {\mathcal Sm}_k$, and we have two open subschemes $U$ and $V$ of $X$ such that $X = U \cup V$.  There are Mayer-Vietoris long exact sequences in $\aone$-homology:
\[
\cdots \longrightarrow H_{i+1}^{\aone}(U \cup V) \longrightarrow H_i^{\aone}(U \cap V) \longrightarrow H^{\aone}_i(U) \oplus H_i^{\aone}(V) \longrightarrow H_i^{\aone}(X) \longrightarrow \cdots,
\]
and reduced $\aone$-homology:
\[
\cdots \longrightarrow \tilde{H}_{i+1}^{\aone}(X) \longrightarrow \tilde{H}_i^{\aone}(U \cap V) \longrightarrow \tilde{H}^{\aone}_i(U) \oplus \tilde{H}_i^{\aone}(V) \longrightarrow \tilde{H}_i^{\aone}(X) \longrightarrow \cdots.
\]
\end{prop}

\begin{proof}
A Zariski covering by two open sets gives rise to a push-out square of the form
\[
\xymatrix{
U \cap V \ar[r]\ar[d] & U \ar[d] \\
V \ar[r] & X
}.
\]
Thus, we get a short exact sequence of sheaves of abelian groups of the form ({\em cf.} \cite{MV} Section 3 Remark 1.7)
\[
0 \longrightarrow \Z(U \cap V) \longrightarrow \Z(U) \oplus \Z(V) \longrightarrow \Z(X) \longrightarrow 0.
\]
Now, the $\aone$-localization functor is exact by Lemma \ref{lem:localizationexact}, and, as a left adjoint, commutes with finite colimits.  Thus, one obtains a short exact sequence of $\aone$-singular chain complexes, and the middle term can be identified with the direct sum of the $\aone$-singular chain complexes corresponding to $U$ and $V$.  Tracking the unit map gives rise to corresponding short exact sequences for reduced $\aone$-singular chain complexes (the proof is identical to that given in e.g., \cite{Spanier} Chapter 4 \S 6).  In either case, taking homology sheaves gives rise to the required exact sequences of $\aone$-homology sheaves.
\end{proof}

\subsubsection*{Some key computations}
Consider the canonical $\gm$-torsor: 
\begin{equation*}
{\mathbb A}^{n+1} - 0 \longrightarrow {\mathbb P}^n.
\end{equation*}
For $n \geq 2$, the space ${\mathbb A}^{n+1} - 0$ is $\aone$-simply connected.  Thus, for $n \geq 2$, the space ${\mathbb A}^{n+1}$ is the universal $\aone$-covering space of ${\mathbb P}^n$.  This is {\em false} for $n = 1$, and ${\mathbb A}^2 - 0$ has a non-trivial $\aone$-fundamental group.  Thus, not all smooth schemes have geometric $\aone$-universal cover.  

We now recall Morel's computation of some $\aone$-homotopy groups as it provides the template for our results on $\aone$-homotopy groups of toric varieties.    In order to state the result, it is necessary to recall the Milnor-Witt K-theory sheaves introduced by Morel (see \cite{MField} \S 2).  Let ${\mathcal S}$ be a pointed space.  The {\em free strictly $\aone$-invariant sheaf of groups generated by ${\mathcal S}$} is by definition the reduced $\aone$-homology sheaf $\tilde{H}_0^{\aone}({\mathcal S})$.  The sheaf $\underline{\bf{K}}^{MW}_n$ of weight $n$ Milnor-Witt K-theory can be defined to be the free strictly $\aone$-invariant sheaf generated by $\gm^{\wedge n}$ (at least if $n \geq 1$).\footnote{Morel also shows that $\underline{\bf{K}}^{MW}_n$ is the free strongly $\aone$-invariant sheaf of {\em abelian} groups generated by $\gm^{\wedge n}$.}  The following result is one of the main computational achievements of \cite{MField}; we will use this result repeatedly in the sequel.  The proof of the result breaks into a largely formal part (given the coterie of theorems mentioned or proved so far), most of which we reproduce below, and the non-trivial task of identifying the sheaf of groups $\underline{\bf{K}}^{MW}_n$ ``concretely," for which we refer the reader to \cite{MField} \S 2, especially Theorem 2.37. (We implicitly fix base-points in the statement below.)  

\begin{thm}[Morel (\cite{MField} Thm 3.40, Thm 4.13)]
\label{thm:projectivespace}
Suppose $n \geq 2$.  There are canonical isomorphisms
\[
\pi_i^{\aone}({\mathbb A}^{n} - 0) \cong
\begin{cases} 0 & \text{ if } i < n-1 \\
\underline{\bf{K}}_n^{MW} & \text{ if } i = n-1.
\end{cases}
\]
Furthermore, there is a canonical central extension
\[
1 \longrightarrow \underline{\bf{K}}_2^{MW} \longrightarrow \pi_1^{\aone}({\mathbb P}^1) \longrightarrow \gm \longrightarrow 1,
\]
and if $n \geq 2$, there is canonical isomorphisms $\pi_1^{\aone}({\mathbb P}^n) \cong \gm$ and $\pi_{n}^{\aone}({\mathbb P}^n) \cong \underline{\bf{K}}_n^{MW}$.
\end{thm}

\section{Excision results for $\aone$-homotopy groups}
\label{s:excision}
The goal of this section is to prove excision style results for $\aone$-homotopy groups of smooth schemes.  Let us begin by stating the main theorem.

\begin{thm}[$\aone$-Excision]
\label{thm:excision} Let $k$ be an infinite field.  Suppose $X \in {\mathcal Sm}_k$ is $\aone$-connected, and $j: U \hookrightarrow X$ is an
open immersion of an $\aone$-connected scheme whose closed
complement is everywhere of codimension $d \geq 2$.  Fix a base
point $x \in U(k)$.  If furthermore $X$ is $\aone$-$m$-connected,
for $m \geq d-3$, then the canonical morphism
\begin{equation*}
j_*: \pi_i^{\aone}(U,x) \longrightarrow \pi_i^{\aone}(X,x)
\end{equation*}
is an isomorphism for $0 \leq i \leq d-2$ and an epimorphism for $i = d-1$.
\end{thm}

\begin{rem}
The expression $\aone$-$(-1)$-connected, which arises when $d=2$ and $m=-1$, means the scheme is nonempty, but the background assumption that $X$ is $\aone$-connected is of course a strictly stronger hypothesis.  In Theorem \ref{thm:excision}, either the condition that $k$ is infinite or the hypothesis that $U$ be $\aone$-connected is necessary.  Indeed, if $k$ is finite, the hypothesis that $U$ be $\aone$-connected need not be satisfied (see Remark \ref{rem:spacefilling}).  On the other hand if $k$ is infinite, it is expected that this can not happen (see Conjecture \ref{conj:codim2}).  We also refer to \cite{MStable} Theorem 6.4.1 for an excision theorem of this sort for stable $\aone$-homotopy sheaves of groups.  
\end{rem}

In the proof of this result, we will treat the case of the $\aone$-fundamental group separately from higher $\aone$-homotopy groups.  The main reason for this is that the $\aone$-fundamental group need not be abelian and thus one must introduce different techniques for its study.  

\begin{convention}
For the rest of this section, non-abelian sheaves of groups that are not necessarily abelian will be denoted using the letters $G$ or $H$ (or primed versions thereof) and abelian sheaves of groups will be denoted by the letter $A$ (or primed versions thereof).
\end{convention}

As the proof of Theorem \ref{thm:excision} is quite long and will use many of the results of the previous section together with a collection of techniques to be introduced here, let us provide an outline.  We begin by reviewing the ``Cousin resolution" that allows one to relate Nisnevich cohomology groups of a sheaf of abelian groups on a scheme $X$ to the points of $X$; this will involve some constructions from local cohomology theory.  The main result of this discussion are Corollaries \ref{cor:excision-abelian} and \ref{cor:excision-nonabelian}, which show that one has an appropriate excision statement for Nisnevich cohomology with coefficients in a strongly or strictly $\aone$-invariant sheaf of groups.  In order to prove the excision result, we need to relate this kind of cohomology to $\aone$-homotopy groups; this connection is provided by the Postnikov tower (see Theorem \ref{thm:postnikov}) that relates homomorphisms in the categories ${\mathcal Gr}^{\aone}_k$ and ${\mathcal Ab}^{\aone}_k$ ({\em cf.} Paragraphs \ref{entry:strongaoneinvarianceequivalences} and \ref{entry:strictaoneinvarianceequivalences}) to appropriate Nisnevich cohomology groups.  We then use the covariant form of the Yoneda lemma to show that the set of homomorphisms out of a particular strongly (resp. strictly) $\aone$-invariant sheaf of groups, characterizes the sheaf of groups.  

If $A$ is a sheaf of abelian groups, it is relatively easy to construct the Cousin complex; this is discussed quite beautifully in \cite{CTHK} \S 1.  Relating the Cousin complex to Nisnevich cohomology of $A$ currently requires two steps and necessitates that one impose some additional conditions on $A$; these conditions may be axiomatized as in \cite{CTHK} \S 5.  Roughly speaking, these axioms include {\em Nisnevich excision} and a form of {\em $\aone$-homotopy invariance}, and are satisfied for strictly $\aone$-invariant sheaves of groups (see Definition \ref{defn:strictly} and Lemma \ref{lem:cousin-exist}).  Both find their way into the proof by way of Gabber's (geometric) presentation lemma, which is very nicely discussed in \cite{CTHK} \S 3 (especially Theorem 3.1.1), and thus necessitate that the base field $k$ be infinite. One begins by proving that the Cousin complex is a flasque resolution of the {\em Zariski} sheaf $A$ (see Theorem \ref{thm:cousin-abelian}).  Next, one proves that the Nisnevich and Zariski cohomology of $A$ co-incide (see Theorem \ref{thm:abeliancompare}).

If $G$ is a non-abelian sheaf of groups, one cannot expect a Cousin complex for Zariski cohomology of $G$ to exist in general.  Nevertheless, Morel constructs in \cite{MField} \S 1.2 a truncated version of the Cousin complex that is refined enough to compute $H^1_{Nis}(X,G)$.  Again, these results are housed in an axiomatic framework involving an appropriate form of Nisnevich excision and $\aone$-homotopy invariance.  Gabber's presentation Lemma once more appears in the proof and necessitates the assumption that $k$ be infinite.  Strongly $\aone$-invariant sheaves of groups (see Definition \ref{defn:stronglyinvariant} and Theorem \ref{thm:stronga1}) satisfy both of the corresponding conditions (in fact Morel gives an axiomatic characterization of such sheaves).  Roughly, one constructs a flasque Zariski ``resolution" of $G$ and uses it to prove a comparison theorem between Zariski and Nisnevich cohomology (see Theorem \ref{thm:compare-nab}).

\subsubsection*{Cousin resolutions}
In this section, we review some aspects of the beautiful axiomatic framework introduced in \cite{CTHK} for construction of Cousin resolutions for sheaves of abelian groups.  We refer the reader to \cite{HartshorneRD} Chapter 4 Proposition 2.6 for a discussion of  the classical Cousin complex (associated with a filtration of a topological space by closed subspaces) and associated conditions guaranteeing that the corresponding complex is a resolution.

Suppose $X$ is a smooth scheme.  We refer the reader to \cite{HartshorneRD} Chapter IV for a discussion of cohomology with supports.  Suppose $X$ is equidimensional and $Z_{p+2} \subset Z_{p+1} \subset Z_{p} \subset X$ is a nested sequence of closed immersions with each $Z_i$ having codimension $\geq i$ in $X$.  The long exact sequence of cohomology with supports for the triple $(X,Z_{p},Z_{p+1})$ gives rise to a connecting homomorphism
\[
H^{i}_{Z_p - Z_{p+1}}(X - Z_p,A) \longrightarrow H^{i+1}_{Z_{p+1}}(X,A)
\]
and similarly, using the triple $(X,Z_{p+1},Z_{p+2})$, one obtains a morphism
\[
H^{i}_{Z_{p+1}}(X,A) \longrightarrow H^i_{Z_{p+1} - Z_{p+2}}(X - Z_{p+2},A);
\]
we write
\[
d^{p,i}: H^{i}_{Z_p - Z_{p+1}}(X - Z_p,A) \longrightarrow H^i_{Z_{p+1} - Z_{p+2}}(X - Z_{p+2},A)
\]
for the composite of these two morphisms.  If we order the collection sequences of closed immersions $\bar{Z} = \setof{Z_{p+2} \subset Z_{p+1} \subset Z_{p}}$ by saying $\bar{Z} \leq \bar{Z'}$ if $Z_p \subset Z'_p$ for all $p$, the functoriality of cohomology with supports shows that the construction of the above groups and morphisms is covariant with respect to the ordering so defined.

For a point $x \in X^{(p)}$ and an integer $q \geq 0$, define $H_x^{p+q}(X,A) := \lim_{U \ni x} H^{p+q}_{\bar{x} \cap U}(U,A)$.  With this notation, passing to the limit in the situation of the previous paragraph allows us to construct complexes (see \cite{CTHK} Lemma 1.2.1 and Sequence 1.3)
\begin{equation}
\label{eqn:cousin1}
0 \longrightarrow \coprod_{x \in X^{(0)}} H^q_{x}(X,A) \stackrel{d^{0,q}}{\longrightarrow} \coprod_{x \in X^{(1)}} H^{q+1}_{x}(X,A) \stackrel{d^{1,q}}{\longrightarrow} \cdots \stackrel{d^{m-1,q}}{\longrightarrow} \coprod_{x \in X^{(m)}} H^{q+m}_x(X,A) \stackrel{d^{m,q}}{\longrightarrow}\cdots
\end{equation}
that we shall refer to as Cousin complexes.  In fact, in the sequel, we shall only use these complexes in the case $q = 0$.

The above construction can also be sheafified for the Zariski topology in the sense that the functors
$$
U \mapsto \coprod_{x \in U^{(p)}} H_{x}^{n}(U,A)
$$
produce sheaves on the small Zariski site of $X$ (see \cite{CTHK} Lemma 1.2.2); these sheaves are in fact {\em flasque}.  Indeed, if we denote $i_x: x \hookrightarrow X$, then we can identify the above sheaf with $\coprod_{x \in X^{(p)}} {i_{x}}_* H^n_x(X,A)$.  Replacing the groups in Sequence \ref{eqn:cousin1} by the sheaves just described produces a complex of flasque sheaves:
\begin{equation}
\label{eqn:cousin2}
0 \longrightarrow \coprod_{x \in X^{(0)}} {i_{x}}_*H^q_{x}(X,A) \stackrel{d^{0,q}}{\longrightarrow} \coprod_{x \in X^{(1)}} {i_{x}}_*H^{q+1}_{x}(X,A) \stackrel{d^{1,q}}{\longrightarrow} \cdots \stackrel{d^{m-1,q}}{\longrightarrow} \coprod_{x \in X^{(m)}} {i_{x}}_*H^{q+m}_x(X,A) \stackrel{d^{m,q}}{\longrightarrow}\cdots.
\end{equation}
If $A$ is a Nisnevich sheaf, we let ${\mathcal H}^q_{Nis}(A)$ and ${\mathcal H}^q_{Zar}(A)$ denote the Nisnevich and Zariski sheaves corresponding to the presheaves $U \mapsto H^q_{Nis}(U,A)$ or $U \mapsto H^q_{Zar}(U,A)$.  We write ${\mathcal H}^q_{Nis}(X,A)$ or ${\mathcal H}^q_{Zar}(X,A)$ for the corresponding sheaves restricted to the appropriate small site of $X$.  The goal then is to give conditions guaranteeing that the complex of Sequence \ref{eqn:cousin2} is in fact a flasque resolution of the cohomology sheaf ${\mathcal H}^q(X,A)$.  We now recall aspects of the axiomatization of \cite{CTHK} \S 5.

In \cite{CTHK} \S 5.1, the authors introduce an abstract notion of ``cohomology theory with supports" taking values in some abelian category ${\mathcal A}$ (of which the usual cohomology with supports with coefficients in a sheaf is the main example).  Furthermore, they construct a ``complex" level refinement of this notion that they term a ``substratum" (see \cite{CTHK} Definition 5.1.1).  Usual cohomology with supports of an abelian sheaf $A$ (taking values in abelian groups) is certainly of this form; for the complex level refinement, assign to $X$ a fibrant resolution (think injective resolution) of the global sections of $A$ (with supports).  We will only consider functors sending a pair $(X,Z)$ consisting of a $X \in {\mathcal Sm}_k$ and a closed subscheme $Z \subset X$ to the cohomology group $H^q_Z(X,A)$.

Recall that a distinguished square is a Cartesian diagram of the form
$$
\xymatrix{
\pi^{-1}(U) \ar[r] \ar[d] & X' \ar[d]^{\pi} \\
U \ar[r]^{j} & X
}
$$
where $\pi$ is \'etale and $\pi$ is an isomorphism from the closed complement $Z'$ of $\pi^{-1}(U)$ in $X$ onto the closed complement $Z$ of $U$ in $X$.  We now recall a pair of axioms introduced in \cite{CTHK} \S 5.  Consider the functor on pairs $(Z,X) \mapsto H^q_Z(X,A)$.

\begin{enumerate}
\item({\bf Nisnevich Excision}) Given any distinguished square, the induced morphism $H^q_{Z}(X) \longrightarrow H^q_{Z'}(X')$ is an isomorphism for all $q$.
\item({\bf Weak $\aone$-homotopy invariance}) Given any open subset $V \subset {\mathbb A}^k$ and any closed subscheme $S$ of $V$, let $\pi: {\mathbb A}^1_V \longrightarrow V$ denote the canonical projection morphism.  The pull-back morphism $H^q_{S}(V,A) \longrightarrow H^q_{S_{\aone}}(V_{\aone},A)$ is an isomorphism for all $q$.
\end{enumerate}

The first key result on Cousin complexes can be summarized as follows; we paraphrase \cite{CTHK} Corollary 5.1.11.

\begin{thm}[\cite{CTHK} Corollary 5.11]
\label{thm:cousin-abelian}
Suppose $k$ is an infinite field. If $A$ is an abelian sheaf such that the functor on pairs $(Z,X) \mapsto H^i_{Z}(X,A)$ satisfies Nisnevich excision and weak $\aone$-homotopy invariance, for any smooth scheme $X$, then the Cousin complexes of Sequence \ref{eqn:cousin2} are flasque resolutions of the Zariski sheaves ${\mathcal H}^q(X,A)$.
\end{thm}

However, our main interest will be in Nisnevich cohomology, not Zariski cohomology.  For this, a further comparison theorem is required.  We write ${\mathcal H}^i_{Zar}(A)$ (resp. ${\mathcal H}^i_{Nis}(A)$) for the sheaf on the big Zariski (resp. Nisnevich) site of $\Sm_k$ associated with the presheaf $U \mapsto H^q(U,A)$.  The content of Theorem 8.3.1 of \cite{CTHK} can then be summarized by the following result.

\begin{thm}[\cite{CTHK} Theorem 8.3.1]
\label{thm:abeliancompare}
Suppose $k$ is an infinite field.  If $A$ is an abelian sheaf satisfying the hypotheses of Theorem \ref{thm:cousin-abelian}, then for any $X \in {\mathcal Sm}_k$, and any $q$, the comparison morphism
$$
H^i_{Zar}(X,{\mathcal H}^q_{Zar}(A)) \longrightarrow H^i_{Nis}(X,{\mathcal H}^q_{Nis}(A))
$$
is an isomorphism for all $i \geq 0$.
\end{thm}

Using the above, one can show that strictly $\aone$-invariant sheaves of groups admit Cousin resolutions; this is summarized in the following result.

\begin{lem}
\label{lem:cousin-exist}
Suppose $k$ is an infinite field.  If $A$ is a strictly $\aone$-invariant sheaf of groups, then the functor on pairs $(X,Z) \mapsto H^0_Z(X,A)$ satisfies Nisnevich excision and weak $\aone$-homotopy invariance.  Thus, for any smooth scheme $X$, the sheaf $A$ admits a Cousin resolution of the form
\[
0 \longrightarrow A \longrightarrow \coprod_{x \in X^{(0)}} {i_{x}}_*H^0_{x}(X,A) \longrightarrow \coprod_{x \in X^{(1)}} {i_{x}}_*H^1_{x}(X,A) \longrightarrow \cdots.
\]
\end{lem}

\begin{proof}
Nisnevich excision follows from the characterization of Nisnevich sheaves in terms of distinguished triangles.  Homotopy invariance follows immediately from the definition of strict $\aone$-invariance.
\end{proof}

For our purposes, the most important observation to make is that in order to study the $i$-th Zariski cohomology group of the sheaf $A$ it suffices to concentrate on the portion of the complex corresponding to points of codimension $i-1$ through $i+1$.

\begin{cor}
\label{cor:excision-abelian} Suppose $k$ is an infinite field and
$A$ is a strictly $\aone$-invariant sheaf of groups.  Suppose $X \in {\mathcal Sm}_k$, and $U \subset X$ is an open subscheme whose
complement is everywhere of codimension $d$.  The restriction
morphism
\[
H^i_{Nis}(X,A) \longrightarrow H^i_{Nis}(U,A)
\]
is a monomorphism for $i \leq d-1$ and an isomorphism for $i \leq d-2$.
\end{cor}

\begin{proof}
Since Zariski and Nisnevich cohomology of a strictly $\aone$-invariant sheaf co-incide by Lemma \ref{lem:cousin-exist} and Theorem \ref{thm:abeliancompare}, it suffices to prove the result for Zariski cohomology.  In this case, again by Lemma \ref{lem:cousin-exist}, the Cousin resolutions of $A$ on $X$ and $U$ are flasque resolutions; let us denote these resolutions by $C^*(X,A)$ and $C^*(U,A)$.   We obtain, by functoriality of the Cousin resolution, a canonical restriction morphism $C^*(X,A) \longrightarrow C^*(U,A)$.  Since $U$ has complement of codimension $d$ in $X$, it follows that the induced map on sets of points of codimension $\geq d-1$ is an isomorphism and a monomorphism on those of codimension $d$.  Thus, the induced morphism of Cousin complexes is an isomorphism in degrees $\leq d-1$ and a monomorphism in degrees $\leq d$.  It follows from the long exact sequence in cohomology that the restriction maps are isomorphisms in degrees $\leq d-2$ and monomorphisms in degrees $\leq d-1$
\end{proof}

\subsubsection*{Truncated Cousin resolutions}
In order to study strongly $\aone$-invariant sheaves of groups $G$, one must work harder.  The basic problem is the usual one in non-abelian cohomology: $H^i(X,G)$ is naturally defined only for $i = 0,1$. It is a group for $i = 0$, but only a pointed set if $i = 1$; this will be reflected in the ``resolution" one constructs.  Morel's idea in \cite{MField} \S 1.2 is to produce a ``truncated Cousin resolution" using an axiomatic approach similar to the above.  His key results provide analogs of Theorem \ref{thm:cousin-abelian} and Theorem \ref{thm:abeliancompare} in the non-abelian situation.  One may then use similar arguments to those above to deduce the comparison theorem for Nisnevich and Zariski cohomology  (analogous to Corollary \ref{cor:excision-abelian}) in the non-abelian setting.  Let us begin by discussing analogs of the first few terms of the Cousin complex for non-abelian groups.

First, in order to control the degree $0$ part of the Cousin complex we introduce an axiom following Morel (see \cite{MField} Defn. 1.1 (0),(1)).  Strictly speaking, this axiom is unnecessary to define the Cousin complex, but will aid us in studying the Cousin resolution.

\begin{enumerate}
\item[{\bf (C0)}]  For any $X \in {\mathcal Sm}_k$ with irreducible components $X_i$ the map $G(X) \longrightarrow \prod_{i \in X^{(0)}}G(X_i)$ is a bijection and for any open, everywhere dense subscheme $U \subset X$, the restriction morphism $G(X) \longrightarrow G(U)$ is injective.
\end{enumerate}

Suppose $x$ is a codimension $1$ point of an irreducible $X \in {\mathcal Sm}_k$.  If $A$ is an abelian sheaf, recall that $H^1_{x}(X,A)$ is by definition $\lim_{U \ni x} H^1_{\bar{x} \cap U}(U,A)$.  Let $X_x = \Spec \O_{X,x}$, $j: X_x \hookrightarrow X$, and $A_x = j^*A$.  We can then identify $H^1_{x}(X,A) = H^1_{x}(X_x,A_x)$.   This set fits into an exact sequence of the form
\[
\cdots \longrightarrow A(\O_{X,x}) \longrightarrow A(F) \longrightarrow H^1_x(X_x,A_x) \longrightarrow H^1(X_x,A_x) \longrightarrow \cdots,
\]
where $F$ is the function field of $X$.  The last group classifies Zariski locally trivial $A$-torsors on $X_x$, but since $X_x$ is a local scheme, all such torsors are trivial.  Using this, one can identify $H^1_x(X_x,A_x)$ with the quotient set $A(F)/A(\O_{X,x})$.  For a Zariski sheaf of groups $G$, observe that
\[
H^1_{x}(X,G) = G(F)/G(\O_{X,x})
\]
is no longer a group, but only a pointed set.  More generally, for a codimension $1$ point $x$ of (a not necessarily irreducible) $X \in {\mathcal Sm}_k$, observe that $H^1_{x}(X,G)$ is well-defined as a pointed set, e.g., by using the exact sequence in cohomology with supports.

Recall that the co-product in the category of pointed sets is the restricted product, i.e., the subset of the direct product consisting of sequences of elements all but finitely many of which are given by the distinguished point.  The above long exact sequence induces (as before) a quotient ``boundary" homomorphism
\[
\coprod_{x \in X^{(0)}} H^0_x(X,G) \Longrightarrow \prod_{x \in X^{(1)}} H^1_x(X,G).
\]
We use the symbol ``$\Longrightarrow$" to denote this quotient map.  In case $G$ is abelian, in order to reduce to the ordinary Cousin complex, the image of this boundary homomorphism must be contained in the restricted product and its kernel must be trivial, we can again impose this as an axiom (see \cite{MField} Defn. 1.1 (2)).  In conjunction with {\bf (C0)}, this axiom gives us the required condition.

\begin{itemize}
\item[{\bf (C1)}] For any irreducible $X \in {\mathcal Sm}_k$, the map $G(X) \longrightarrow \cap_{x \in X^{(1)}} G(\O_{X,x})$ (intersection taken in $G(F)$) is a bijection.
\end{itemize}

As in the abelian case, we would like to ``extend our sequence further to the right."  Using the description of the differential in the Cousin complex (see Sequence \ref{eqn:cousin1}) we can construct, for non-abelian $G$, a target set for the ``differential" mapping from the term corresponding to codimension $1$ points and related to codimension $2$ points.  To do this, suppose $x$ is a codimension $2$ point of $X$ and consider the set $X_x^{(1)}$ of all codimension $1$ points $y \in X$ such that $y \in X_x$.  Let $F$ be the function field of $X_x$ and observe that $G(F)$ acts on $H^1_y(X,G)$ for such $y$.  In general, there is no reason for $G(F)$ to preserve the restricted product $\coprod_{y \in X_x^{(1)}} H^1_{y}(X,G)$, but, assuming $G$ satisfies ${\bf (C1)}$, this is true.  Set
\[
H^2_{x}(X,G) = \coprod_{y \in X_x^{(1)}} H^1_{y}(X,G)/G(F).
\]

The composite maps of pointed sets
\[
{\coprod}_{x \in X^{(1)}} H^1_x(X,G) \longrightarrow {\prod}_{x \in X_z^{(1)}} H^1_x(X,G) \longrightarrow H^2_z(X,G)
\]
fit together to give a $G(F)$-equivariant ``boundary homomorphism"
\begin{equation}
\label{eqn:boundary}
{\coprod}_{x \in X^{(1)}} H^1_{x}(X,G) \longrightarrow {\prod}_{x \in X^{(2)}} H^2_x(X,G).
\end{equation}
Again, it is not clear that the image of this map is actually contained in the co-product.  Once more, following Morel (see \cite{MField} Axiom (A2') p. 24), we introduce an axiom to deal with this problem.

\begin{enumerate}
\item[{\bf (C2)}] For any $X \in {\mathcal Sm}_k$, the image of the boundary map in Map \ref{eqn:boundary} is contained in the coproduct $\coprod_{x \in X^{(2)}} H^2_x(X,G)$
\end{enumerate}

Summarizing, if $G$ is any Zariski sheaf of groups satisfying {\bf (C0)}, {\bf (C1)} and {\bf (C2)}, then we obtain a sequence of groups and pointed sets of the form
\[
\coprod_{x \in X^{(0)}} H^0_x(X,G) \Longrightarrow \coprod_{x \in X^{(1)}} H^1_x(X,G) \longrightarrow \coprod_{x \in X^{(2)}} H^2_{x}(X,G).
\]
The last arrow is a $G(F)$-equivariant homomorphism as well; this sequence will play the role of the Cousin complex of Sequence \ref{eqn:cousin1}.

We can also sheafify this construction as in Sequence \ref{eqn:cousin2}.  Indeed, for $i = 0,1,2$, the functors $U \mapsto \coprod_{x \in U^{(i)}} H_{x}(U,G)$ give rise to flasque Zariski sheaves on the big Zariski site;  denote the corresponding collection of Zariski sheaves by $G^{(i)}$ ({\em cf.} after Remark 1.19 of \cite{MField}).

Using the injective morphisms $G(X) \longrightarrow \coprod_{x \in X^{(0)}} H_x(X,G)$ and $G \longrightarrow G^{(0)}$, we obtain sequences
\begin{equation}
\label{eqn:cousin-na1}
1 \longrightarrow G(X) \longrightarrow \coprod_{x \in X^{(0)}} H^0_x(X,G) \Longrightarrow \coprod_{x \in X^{(1)}} H^1_x(X,G) \longrightarrow \coprod_{x \in X^{(2)}} H^2_{x}(X,G),
\end{equation}
and corresponding sequences of Zariski sheaves
\begin{equation}
\label{eqn:cousin-na2}
1 \longrightarrow G \longrightarrow G^{(0)} \Longrightarrow G^{(1)} \longrightarrow G^{(2)}.
\end{equation}
We will refer to these sequences as {\em truncated Cousin complexes}. In order to extend the above constructions to produce a Cousin {\em resolution} one needs to introduce an appropriate notion of exactness and relate the above sequences to Zariski cohomology of $G$ (at least in degrees $0$ and $1$).  Let us first deal with exactness.

\begin{defn}[\cite{MField} Definition 1.20]
\label{defn:exact}
We will say that Sequence \ref{eqn:cousin-na1} (or stalkwise, Sequence \ref{eqn:cousin-na2}) is {\em exact} if
\begin{enumerate}
\item the group $G(X)$ is the isotropy subgroup of the base-point in $\coprod_{x \in X^{(1)}} H^1_x(X,G)$ for the action of $\coprod_{x \in X^{(0)}} H^0_x(X,G)$, and
\item the ``kernel" of the second boundary map (of pointed sets) is equal to the orbit under $\coprod_{x \in X^{0}} H^0_x(X,G)$ of $\coprod_{x \in X^{(1)}} H^1_x(X,G)$.
\end{enumerate}
If a truncated Cousin complex satisfies the conditions of Definition \ref{defn:exact}, we will refer to it as a {\em truncated Cousin resolution.}
\end{defn}

The relation of the above complexes with Zariski cohomology can be understood as follows.  Denote by ${\mathcal Z}^1(\cdot,G)$ the {\em sheaf-theoretic} orbit under $G^{(0)}$ of the base-point in $G^{(1)}$ .  Precisely, this is the sheaf associated with the pre-sheaf whose sections over $U \in {\mathcal Sm}_k$ are the elements of the orbit under $G^{(0)}(U)$ of the base-point in $G^{(1)}(U)$.  We thus obtain an exact (in the same sense as above) sequence of sheaves
\[
1 \longrightarrow G \longrightarrow G^{(0)} \Longrightarrow {\mathcal Z}^1(\cdot,G) \longrightarrow *
\]
As $G^{(0)}$ is flasque, $H^1_{Zar}(X,G^{(0)})$ vanishes.  Thus, we can identify $H^1_{Zar}(X,G)$ with the quotient pointed set ${\mathcal Z}^1(X,G)/G^{(0)}(X)$.  Note that this identification does not involve the ``$H^2$" term of the Cousin complex.  

If $G$ is a Nisnevich sheaf of groups, then it is also a Zariski sheaf of groups, so the above discussion applies.  For such a $G$, let ${\mathcal K}^1(\cdot,G)$ denote the kernel of the boundary morphism $G^{(1)} \longrightarrow G^{(2)}$; this is {\em a priori} a Zariski sheaf of pointed sets.  There is an obvious injective morphism ${\mathcal Z}^1(\cdot,G) \hookrightarrow {\mathcal K}^1(\cdot,G)$.  Furthermore, this monomorphism is $G^{(0)}$-equivariant.  Shortly, we will relate ${\mathcal K}^1(\cdot,G)$ to $H^1_{Nis}(X,G)$.  In order to do this, however, it is important to know that ${\mathcal K}^1(\cdot,G)$ is in fact a Nisnevich sheaf; this will be related to a form of Nisnevich excision for $G$.

Observe that condition (1) of Definition \ref{defn:exact} follows from {\bf (C1)}.  To impose condition (2) of the same definition, Morel introduces analogs of the Nisnevich excision and weak $\aone$-homotopy invariance conditions above.

\begin{enumerate}
\item ({\bf Weak Nisnevich Excision})
\begin{itemize}
\item[{\bf a)}] Suppose $X$ and $X'$ are localizations of smooth $k$-schemes at points of  codimension $1$ with closed points $x$ and $x'$.  If $f: X' \longrightarrow X$ is a local \'etale morphism inducing an isomorphism of closed points, then the induced morphism $H^1_x(X) \longrightarrow H^1_{x'}(X)$ is an isomorphism.
\item[{\bf b)}] Suppose $X$ and $X'$ are localizations of smooth $k$-schemes at points of codimension $2$ with closed points $x$ and $x'$.  If $f: X' \longrightarrow X$ is a local \'etale morphism inducing an isomorphism on closed points, then map of pointed sets $H^2_{x}(X,G) \longrightarrow H^2_{x'}(X',G)$ has trivial kernel.
\end{itemize}
\item ({\bf Weak $\aone$-homotopy invariance}) Suppose $X$ is the localization of a smooth $k$-scheme at a point of codimension $\leq 1$.  The morphism $G(X) \longrightarrow G(\aone_{X})$ is a bijection and the Cousin complex
\[
1 \longrightarrow G(\aone_{X}) \longrightarrow \coprod_{x\in {\aone_{X}}^{(0)}} H^0(\aone_{X},G) \Longrightarrow \coprod_{x \in {\aone_X}^{(1)}} H^1_{x}(\aone_{X},G) \longrightarrow \coprod_{x \in {\aone_X}^{(2)}} H^2_{x}(\aone_{X},G)
\]
is exact.
\end{enumerate}

\begin{lem}[\cite{MField} Lemma 1.24]
If $G$ is a Nisnevich sheaf of groups satisfying {\bf (C0)}, {\bf (C1)}, {\bf (C2)} and weak Nisnevich excision, then ${\mathcal K}^1(\cdot,G)$ is a Nisnevich sheaf.  In this case, we have for any $X \in {\mathcal Sm}_k$ a (functorial) bijection
\[
{\mathcal K}^1(X,G)/G^{(0)}(X) \isomto H_{Nis}^1(X,G).
\]
\end{lem}

\begin{thm}[\cite{MField} Theorem 1.26]
\label{thm:compare-nab}
Let $k$ be an infinite field.  If $G$ is a Nisnevich sheaf of groups satisfying {\bf (C0)}, {\bf (C1)}, {\bf (C2)}, weak Nisnevich excision and weak $\aone$-homotopy invariance, then for any $X \in {\mathcal Sm}_k$, the canonical comparison map
\[
H^1_{Zar}(X,G) \longrightarrow H^1_{Nis}(X,G)
\]
is (functorially in $X$ and $G$) a bijection.
\end{thm}

Finally, the key point is that strongly $\aone$-invariant sheaves satisfy the hypotheses of the previous theorem.

\begin{thm}[\cite{MField} Theorem 3.9]
\label{thm:stronga1}
Let $k$ be an infinite field.  If $G$ is a strongly $\aone$-invariant sheaf of groups, then $G$ satisfies {\bf (C0)}, {\bf (C1)}, {\bf (C2)}, weak Nisnevich excision and weak $\aone$-homotopy invariance.
\end{thm}

As before, we can deduce our excision results for $H^1$ relatively easily from these facts.

\begin{cor}
\label{cor:excision-nonabelian}
Let $k$ be an infinite field.  Suppose $X \in {\mathcal Sm}_k$, and $U \subset X$ is an open subscheme whose complement is of codimension $\geq d$.  If $G$ is any strongly $\aone$-invariant sheaf of groups, then the restriction morphism
\[
H^1_{Nis}(X,G) \longrightarrow H^1_{Nis}(U,G)
\]
is a monomorphism for $d \geq 2$ and is an isomorphism for $d \geq 3$.
\end{cor}

\begin{proof}
Since $G$ is a strongly $\aone$-invariant sheaf of groups, we know that Zariski and Nisnevich cohomology of $G$ co-incide for any smooth scheme $X$ by Theorem \ref{thm:compare-nab}.  If $U$ has complement of codimension $\geq 2$ in $X$, then it follows that ${\mathcal K}^1(X,G) \longrightarrow {\mathcal K}^1(U,G)$ is a monomorphism by definition.  Similarly, if $U$ has complement of codimension $\geq 3$, then it follows that ${\mathcal K}^1(X,G)$ and ${\mathcal K}^1(U,G)$ co-incide as they only depend on points of $X$ of codimension at most $2$.
\end{proof}

\subsubsection*{Proof of Theorem \ref{thm:excision}}
\begin{proof}[Proof of Theorem \ref{thm:excision}]
Suppose $X \in {\mathcal Sm}_k$.  Assume that $(X,x)$ is a pointed $\aone$-$m$-connected scheme, and that $U \subset X$ is an $\aone$-connected open subscheme (pointed by $x$) whose complement is of codimension $d \geq 2$.  For any strongly $\aone$-invariant sheaf of groups $G$, the inclusion map $U \subset X$ gives rise to a commutative square of the form
\begin{equation}
\label{eqn:torsorpullback}
\xymatrix{
[(X,x),(BG,\ast)]_{\aone}\ar[r]\ar[d] & [(U,x),(BG,\ast)]_{\aone}\ar[d] \\
Hom_{{\mathcal Gr}^{\aone}_k}(\pi_1^{\aone}(X,x),G) \ar[r]& Hom_{{\mathcal Gr}^{\aone}_k}(\pi_1^{\aone}(U,x),G).
}
\end{equation}
Since both $U$ and $X$ are $\aone$-connected, Theorem \ref{thm:postnikov} shows that both vertical maps are bijections.  

Suppose that $(Y,y)$ is an arbitrary pointed $\aone$-connected smooth scheme.  Corollary \ref{cor:torsordescription} allows us to identify $[(Y,y),(BG,\ast)]_{\aone}$ with the set of isomorphism classes of pairs consisting of a $G$-torsor on $Y$ together with an element $g \in G(k)$.  Equivalently, using the discussion of Paragraph \ref{entry:strongaoneinvarianceequivalences}, we can identify $[(Y,y),(BG,\ast)]_{\aone}$ with the set of pairs $({\mathcal P},g)$ consisting of an element ${\mathcal P} \in H^1_{Nis}(Y,G)$ and an element $g \in G(k)$.  

Returning to the situation of Diagram \ref{eqn:torsorpullback}, if the codimension of $X \setminus U$ is at least $3$, then we use Corollary \ref{cor:excision-nonabelian} to conclude that, for any strongly $\aone$-invariant sheaf of groups $G$, any $G$-torsor on $U$ extends, up to isomorphism, to a $G$-torsor on $X$.  Thus, assuming $\codim X \setminus U \geq 3$, we conclude that the upper horizontal map of Diagram \ref{eqn:torsorpullback} is a bijection for every strongly $\aone$-invariant sheaf of groups $G$.  Similarly, if $\codim X \setminus U \geq 2$, then we conclude again using Corollary \ref{cor:excision-nonabelian} that the upper horizontal map of Diagram \ref{eqn:torsorpullback} is, functorially in $G$, a monomorphism.  

We conclude that for any strongly $\aone$-invariant sheaf of groups $G$, the lower horizontal map of Diagram \ref{eqn:torsorpullback} is an isomorphism if $\codim X \setminus U \geq 3$ or a monomorphism if $\codim X \setminus U \geq 2$, functorially in $G$.  The covariant form of the Yoneda lemma shows that the morphism $\pi_1^{\aone}(U,x) \longrightarrow \pi_1^{\aone}(X,x)$ is an isomorphism in the first situation and an epimorphism in the second situation.  

The higher dimensional cases follow similarly.  If $X$ is $\aone$-$m$-connected for $m > 1$, and $\codim X \setminus U = d \geq 3$, then we can conclude inductively that $U$ is $\aone$-$k$-connected, for $k = \min(m,d-3)$ using Corollary \ref{cor:excision-abelian} (in place of Corollary \ref{cor:excision-nonabelian}) together with Theorem \ref{thm:postnikov}.  
\end{proof}

\section{Geometric quotients, $\aone$-covers, and toric varieties}
\label{s:geometry}
In this section, we discuss the $\aone$-covering
space theory introduced in \S \ref{s:higher} in the context of
geometric invariant theory for solvable group actions.  The
motivating principle behind this relation is Proposition
\ref{prop:solvablequot}, which shows how geometric invariant
theory (GIT) for solvable group actions may be used to construct geometric
Galois $\aone$-covering spaces (recall Definition \ref{defn:geometriccovering}).  The simplest examples to which the
theory applies are the complete flag varieties $SL_n \rightarrow
SL_n/B$, for which, as far as GIT is concerned, every point is
``stable" and so nothing need be excised.

Concrete computations of $\aone$-homotopy groups of such quotients $X$ in these instances are quite difficult, largely because we don't know enough about the $\aone$-homotopy groups of the source variety.  The better strategy is then to start with a space whose $\aone$-homotopy type one understands, e.g., $\mathbb{A}^n$, and consider GIT-style quotients by solvable groups.  We studied quotients for free unipotent actions on affine space in \cite{AD1}.  On the other hand, quotients of split torus actions on affine space yield toric varieties.

After motivating the construction of toric varieties via GIT for split torus actions, we recall Cox's description (see \cite{Cox}) of {\em any} simplicial (in particular smooth) toric variety as a geometric quotient of an open subset of affine space by the free action of a split torus.  This will allow for a broader range of open sets than would arise as the GIT-stable
points for a linearized split torus action on affine space; for
instance, one can produce non-quasi-projective smooth toric
varieties. Along the way, we will give a quick summary of the geometry
and combinatorics of toric varieties relevant to the discussion of
subsequent sections, and establish key Propositions \ref{prop:combinatorics} and \ref{prop:galoiscover}.

\subsection*{Solvable quotients and $\aone$-covers}
The motivating observation for the results in this paper is summarized in the following result.

\begin{prop}
\label{prop:solvablequot}
Let $G$ be a connected, split, solvable algebraic group. Suppose $G$ acts freely on a smooth scheme $X$ such that i) a quotient $q: X \longrightarrow X/G$ exists as a smooth scheme, ii) the triple $(X,q,G)$ is a $G$-torsor. If both $X$ and $X/G$ are $\aone$-connected, then the morphism $q$ makes $X$ into a (homotopy) geometric Galois $\aone$-covering space of $X/G$.
\end{prop}

\begin{proof}
Since $G$ is $k$-split, its unipotent radical $R_u(G)$ is
$k$-defined as well; the quotient $G/R_u(G)$ is a split torus $T$.
Consider the quotient $X/R_u(G)$, which exists as a scheme.  The quotient morphism $X \longrightarrow X/R_u(G)$ is an $R_u(G)$-torsor and is thus an $\aone$-weak equivalence (being Zariski locally trivial with fibers
isomorphic to affine space; see e.g., \cite{AD1} Lemma 3.3).
Observe then that $T$ acts on $X/R_u(G)$ and the quotient morphism
$X/R_u(G) \longrightarrow X/G$ is a $T$-torsor.  If both $X$ and $X/G$ are assumed $\aone$-connected, it follows that $X/R_u(G) \longrightarrow X/G$ is a geometric Galois $\aone$-covering space.  
\end{proof}

\begin{rem}
Note that as discussed in the proof, if $G$ is a split unipotent
group then $X \longrightarrow X/G$ is not merely a (homotopy)
$\aone$-covering space but in fact is an $\aone$-weak equivalence.
This is a handy way of producing moduli of schemes with fixed
$\aone$-homotopy type (see \cite{AD1}).
\end{rem}

\begin{cor}
\label{cor:homotopycovering}
Work under the assumptions of Proposition \ref{prop:solvablequot}. Let $x \in X(k)$, and $q(x)$ be the corresponding $k$-rational point of $X/G$.  The quotient morphism induces a canonical extension:
$$
1 \longrightarrow \pi_1^{\aone}(X,x) \longrightarrow \pi_1^{\aone}(X/G,q(x)) \longrightarrow G/R_u(G) \longrightarrow 1,
$$
and isomorphisms $q_*: \pi_i^{\aone}(X,x) \isomto \pi_i^{\aone}(X/G,q(x))$ for all $i \geq 2$.
\end{cor}

\begin{proof}
Up to $\aone$-weak equivalence, $q$ is an $\aone$-fibration.  Thus, we get a long exact sequence in homotopy groups for an $\aone$-fibration:
$$
\cdots \longrightarrow \pi_i^{\aone}(G/R_u(G)) \longrightarrow \pi_i^{\aone}(X,x) \stackrel{q_*}{\longrightarrow} \pi_i^{\aone}(X/G,q(x)) \longrightarrow \pi_{i-1}^{\aone}(G/R_u(G)) \longrightarrow \cdots
$$
Now, $\gm^{\times r}$ is $\aone$-rigid (see Example \ref{ex:aonerigid}) which shows that $\pi_0^{\aone}(X) \cong \gm^{\times r}$ and all higher $\aone$-homotopy groups vanish.  

%Similarly,
%$$[S^i_s \wedge U_+,(X,x)]_{\aone} \cong [S^i_s \wedge U_+,(\gm^{\times r},1)]_{\aone}.$$
%Since $\gm^{\times r}$ is already strongly-$\aone$-invariant, we just have to compute the group $[S^i_s \wedge U_+,(\gm,1)]_s$.  Since we have a bijection $Hom_{\Sm_k}(\Delta^i_s \times U,\gm) \isomto Hom_{\Sm_k}(U,\gm)$ it follows that the last group vanishes if $i > 0$.
\end{proof}

\begin{ex}
If $G = SL_n$ then $G \longrightarrow G/B$ is an $\aone$-covering
space, inducing an isomorphism on $\pi_i^{\aone}$ for all $i \neq 1$.
This follows from Proposition \ref{prop:solvablequot}, Corollary \ref{cor:homotopycovering}, and the fact that $SL_n$ is $\aone$-connected.  More generally, if $G$ is a connected split reductive linear algebraic group, and $B$ is a Borel subgroup of $G$, then the Bruhat decomposition shows that every point of $G/B$ has a neighborhood isomorphic to affine space (over $k$).  It follows by Lemma \ref{lem:affineneighborhoods} that such $G/B$ are $\aone$-connected.  If $G$ is $\aone$-connected,  then $G \longrightarrow G/B$ is a (homotopy)
$\aone$-covering space and thus $\pi_i^{\aone}(G) \isomto \pi_i^{\aone}(G/B)$ for all $i \neq 1$.  Furthermore, if $T$ denotes a split maximal torus of $B$, then we have a canonical extension
$$
1 \longrightarrow \pi_1^{\aone}(G) \longrightarrow \pi_1^{\aone}(G/B) \longrightarrow T \longrightarrow 1.
$$
Note that a classical result of Steinberg on generation of groups by additive subgroups shows that connected, split, semi-simple, simply connected groups are in fact $\aone$-chain connected.  
\end{ex}

To construct solvable quotients in a manner that fits well with the
excision results we have proved so far, we can use a version of the
geometric invariant theory for non-reductive groups studied in
\cite{DK}.  The essential idea of GIT is to construct nice
``parameter spaces" for orbits under a group action on $X$, known as
``good quotients."  GIT provides a tool, the Hilbert-Mumford
numerical criterion, for identifying open sets $X^s$ of stable
points for which the action is proper and hence the good quotient is
exactly a geometric quotient in the traditional sense.  In nice
enough settings (``stable equals semi-stable") the stable locus is
precisely the complement of the simultaneous vanishing locus for
invariant sections of a chosen $G$-equivariant line bundle on $X$.

\begin{prop}[\cite{DK} Theorem 5.3.1]
Let $G$ be an affine algebraic group. Let $X$ be a smooth
$G$-quasi-projective variety with a chosen $G$-equivariant line
bundle.  There is a canonically determined open set of stable points
$X^s$ on which the action is proper and whose geometric quotient
exists.
\end{prop}

\subsubsection*{GIT for linear torus actions on affine space}

Given a torus $T$ acting linearly on $\mathbb{A}^n$ together with a
choice of $T$-equivariant structure on the trivial line bundle, one
may consider the collection of $T$-equivariant sections (henceforth
called ``invariants").  The locus of geometric points on which all
invariants vanish is called the {\em unstable set} $Z$, which
uniquely defines an unstable (closed) subscheme $s(Z)$ of
$\mathbb{A}^n$. In an appropriate basis, represented by coordinate
functions $\{x_1, \ldots, x_n\}$, the $T$ action is diagonal and the
invariants are generated by monomials.

Because it is defined in terms of geometric points, the unstable set
$Z$ is the vanishing locus for a finite collection of invariant
monomials chosen so that each factor in a given monomial occurs
without repeats. That is, it suffices to consider monomials of the
form $x_{i_1} \cdots x_{i_k}$ where all $i_j$ are distinct.
Formally:

\begin{lem}
The unstable set $Z$ is the zero set for an ideal generated by
finitely many monomials of the form $x_{i_1} \cdots x_{i_k}$ where
the $i_j$ are all distinct.
\end{lem}

In particular the set $Z$ is a coordinate linear subspace arrangement determined by such a set of monomial equations.  One can then ask under what conditions the complementary open subscheme is $\aone$-connected (which is guaranteed when $\codim Z \geq 2$) and the associated restricted $T$-action is free.  Our excision results (Theorem \ref{thm:excision}) together with Proposition \ref{prop:solvablequot} would then apply.

One constraint of this method is that all the quotient varieties so
obtained are necessarily quasi-projective, as they inherit an ample
line bundle by formal properties of GIT.  For this reason and for
the sake of completeness we instead look at the general
combinatorial characterization of $T$-invariant opens of affine
space that yield toric varieties as geometric quotients, due to Cox
(see \cite{Cox}) who in turn was motivated by the traditional approach to
toric varieties via unions of affine toric varieties. Once again the
key data will be sets of monomials whose vanishing describes the
coordinate arrangement $Z$, but here they will be encoded via a
combinatorial device called a {\em fan}.

\subsubsection*{Combinatorial encoding of geometry of toric varieties}
Let $T$ be a split torus over a field $k$.  Let ${\rm X}_*(T)$ and
${\rm X}^*(T)$ denote the co-character and character groups of $T$.
We denote by $\langle \cdot, \cdot \rangle$ the canonical pairing
${\rm X}_*(T) \times {\rm X}^*(T) \longrightarrow \Z$ defined by
composition.  Since $T$ is split, the category of $k$-rational representations of $T$ is semi-simple and every
irreducible representation is given by a character.

A good reference for the theory of toric varieties is \cite{Fulton},
but due to the wealth of differing notation in the field, we felt it
best to summarize our conventions here.  A {\em normal} $T$-variety is
said to be a {\em toric $T$-variety} if $T$ acts on $X$ with an open
dense orbit.  If $T$ is clear from context, we will drop it from the notation.  By a Theorem of Sumihiro, every $k$-point in a toric
variety has a $T$-stable open affine neighborhood; such a variety is
necessarily an affine toric $T$-variety.  Thus, we can cover any
toric variety by affine toric varieties.

The theory of affine toric varieties is particularly simple (see \cite{Fulton} \S 1.3).  If $X$ is an affine toric $T$-variety, we let $k[X] = \Gamma(X,\O_X)$.  The ring $k[X]$ naturally has a $T$-module structure and decomposes as a direct sum of characters.  Since $T$ acts with a dense orbit, it is easy to see that each character in $k[X]$ appears with multiplicity at most $1$.  Since $k[X]$ is a ring, the subset $\Lambda \subset {\rm X}^*(T)$ of characters appearing in the decomposition is in fact a unital monoid.  Furthermore, one can check that it is {\em finitely generated}, the cancellation law holds (i.e., $x + y = x' + y$ in $\Lambda$ implies $x = x'$), and it is {\em saturated} (i.e., if $m$ is an integer and $mx \in \Lambda$, then $x \in \Lambda$).  All of this can be phrased nicely in terms of the co-character lattice ${\rm X}_*(T)$, or rather the associated real vector space $N_{\real} = {\rm X}_*(T) \tensor_{\Z} \real$.  The monoid $\Lambda$ determines a strongly convex rational polyhedral cone $\sigma$ in $N$ and can be uniquely recovered from this data.   We write $X_{\sigma}$ for the affine toric variety associated with a cone $\sigma$.  

\begin{rem}
Let us note here that any {\em smooth} affine toric variety is $T$-equivariantly isomorphic to a product of the form ${\mathbb A}^r \times \gm^{r'}$ (see \cite{Fulton} \S 2.1 Proposition). 
\end{rem}

The dimension of a cone is the cardinality of a minimal set of generators.  A generator $\rho$ of a $1$-dimensional cone is called {\em primitive} if $m\rho' = \rho$ implies $m=1$ and $\rho' = \rho$.  

A {\em fan} $\Sigma$ in ${\rm X}_*(T)$ is a collection of strongly convex rational polyhedral cones $\sigma \in N_{\real}$ such that (i) each face of a cone in $\Sigma$ is a cone in $\Sigma$, and (ii) the intersection of two cones in $\Sigma$ is a face of each.  Henceforth, the word {\em cone} will mean strongly convex rational polyhedral cone.  Any normal toric $T$-variety X gives rise to a fan $\Sigma$: attach to $X$ the cones corresponding to an open cover of $X$ by affine toric $T$-varieties.  Conversely, given a fan $\Sigma$, we can recover a toric $T$-variety which we denote by $X_{\Sigma}$ throughout the sequel.  

We will use the following terminology regarding fans.  We will refer to $\Sigma$ as a {\em smooth fan} if every cone $\sigma \in \Sigma$ is generated by part of a basis for the lattice ${\rm X}_*(T)$.  The support $\Supp(\Sigma)$ of $\Sigma$ is the union of the cones $\sigma \in \Sigma$.  A fan $\Sigma$ will be called {\em proper} if $\Supp(\Sigma) = N_{\real}$.  A {\em refinement} of a fan $\Sigma$ is a fan $\Sigma'$ such that
$\Supp(\Sigma) = \Supp(\Sigma')$ and for every cone of $\sigma' \in
\Sigma'$ there exists a cone $\sigma \in \Sigma$ such that $\sigma'
\subset \sigma$.

\begin{rem}
Any smooth proper toric variety can be covered by affine toric varieties isomorphic to affine space.  
\end{rem}

As the terminology suggests, smooth fans correspond bijectively to
smooth toric varieties and proper fans correspond bijectively to
proper toric varieties.  Refinements of fans correspond to proper
birational morphisms of the corresponding toric varieties.  We will
refer to a fan $\Sigma$ as {\em projective} if $X_{\Sigma}$ is a
projective toric variety.

\subsubsection*{Smooth toric varieties and geometric $\aone$-covers}
If $X_{\Sigma}$ is a toric variety associated with a smooth proper fan $\Sigma$, our goal now is to construct a canonical $\aone$-cover of $X_{\Sigma}$ which we will refer to as the {\em Cox cover} of $X_{\Sigma}$.  To do this, we will show that $X_{\Sigma}$ is a geometric quotient of an open subscheme of an appropriate affine space.

Let $\Sigma(1)$ denote the set of $1$-dimensional cones in $\Sigma$.
Recall that $Pic(X_{\Sigma})$ fits into an exact sequence
$$
0 \longrightarrow {\rm X}^*(T) \longrightarrow \Z^{\Sigma(1)} \longrightarrow Pic(X_{\Sigma}) \longrightarrow 0.
$$
The set $\Sigma(1)$ can be interpreted as the set of $T$-invariant Weil divisors on $X_{\Sigma}$.

The affine space ${\mathbb A}^{\Sigma(1)}$ can be viewed as a toric
variety equipped with an action of the torus $\gm^{\times
\Sigma(1)}$ dual to $\Z^{\Sigma(1)}$.  The above exact
sequence of lattices gives rise, by duality, to an exact
sequence of tori:
$$
0 \longrightarrow Pic(X_{\Sigma})^{\vee} \longrightarrow \gm^{\times \Sigma(1)} \longrightarrow T \longrightarrow 0.
$$
Via this sequence, we can view $Pic(X_{\Sigma})^{\vee}$ as acting on
${\mathbb A}^{\Sigma(1)}$.

Choose coordinates $x_1,\ldots,x_{\Sigma(1)}$ on ${\mathbb
A}^{\Sigma(1)}$.  The $Pic(X_{\Sigma})^{\vee}$ action on ${\mathbb
A}^{\Sigma}$ is equivalent to a $Pic(X_{\Sigma})$-grading on the
polynomial ring $k[{\mathbb A}^{\Sigma(1)}]$.  Since a monomial
$\prod_{\rho} x_\rho^{a_{\rho}}$ determines a divisor $\sum_\rho
a_\rho D_\rho$ (where $D_\rho$ is the coordinate hyperplane defined
by $x_\rho$), the degree of the aforementioned monomial is the image
of $\sum_\rho a_\rho D_\rho$ in $Pic(X_{\Sigma})$.

\begin{defn}
For a cone $\sigma \in \Sigma$, we let $\hat{\sigma}$ be the divisor
$\sum_{\rho \notin \sigma(1)} D_\rho$, and we let $x^{\hat{\sigma}}$
be the monomial $\prod_{\rho \notin \sigma(1)} x_\rho$.  We will
refer to $x^{\hat{\sigma}}$ as the {\em cone-complement monomial}
associated with $\sigma$.
\end{defn}

\begin{defn}[Irrelevant subvariety]
\label{defn:irrelevantideal} Let $\Sigma$ be a proper fan. Let
$Z_{\Sigma}$ be the variety associated with the ideal $I_{\Sigma}$
generated the monomials $x^{\hat{\sigma}}$.  The variety
$Z_{\Sigma}$ will be called the {\em irrelevant subvariety},
corresponding to the fan $\Sigma$.
\end{defn}

By definition, $Z_{\Sigma}$ is a union of coordinate subspaces of
${\mathbb A}^{\Sigma(1)}$; this will be the key property we will use
in the study of the $\aone$-homotopy groups of $X_{\Sigma}$ in what
follows.

\begin{thm}[Cox]
\label{thm:cox}
The group $Pic(X_{\Sigma})^{\vee}$ leaves $Z_{\Sigma}$ invariant and
acts freely on its complement in ${\mathbb A}^{\Sigma(1)}$.
Furthermore, there is a canonical identification $X_{\Sigma} =
({\mathbb A}^{\Sigma(1)} - Z_{\Sigma})/Pic(X_{\Sigma})^{\vee}$.
\end{thm}

\subsubsection*{Combinatorics related to $Z_{\Sigma}$}
Our goal now is to relate the combinatorial structure of $\Sigma$ to
the geometry of the variety $Z_{\Sigma}$.  As we mentioned above,
$Z_{\Sigma}$ is a union of coordinate hypersurfaces in ${\mathbb
A}^{\Sigma(1)}$.  Our goal will be to give conditions that guarantee 1)
that $Z_{\Sigma}$ has codimension $\geq d$, 2) assuming (1) that $Z_{\Sigma}$ has exactly $r$ codimension $d$ components, and 3) assuming (2) that the intersection of any pair of codimension $d$ subspaces in $Z_{\Sigma}$ has codimension $\geq d +2$.  These conditions will form the
combinatorial backbone of our vanishing and non-vanishing results
for $\aone$-homotopy groups in the next section.  In order to do
this, we investigate the ideal defining $Z_{\Sigma}$ in greater
detail.  Recall (Definition \ref{defn:irrelevantideal}) that the
ideal defining $Z_{\Sigma}$ is generated by monomials of the form
$x_{\hat{\sigma}}$.

Cox shows that the variety defined by the ideal $I_{\Sigma}$
generated by the set of cone-complement monomials is precisely the
coordinate subspace arrangement $Z_{\Sigma}$.  The combinatorial
conditions we will require are summarized in the following
proposition.

\begin{prop}
\label{prop:combinatorics}
Suppose $\Sigma$ is a smooth, proper fan.  Suppose $Z_{\Sigma} \subset {\mathbb A}^{\Sigma(1)}$ is the coordinate subspace arrangement associated with the fan $\Sigma$.
\begin{enumerate}
\item The subspace arrangement $Z_{\Sigma}$ has codimension $\geq d$ in ${\mathbb A}^{\Sigma(1)}$ if and only if every (non-degenerate) collection of $d-1$ primitive vectors in $\Sigma$ is part of a cone in $\Sigma$.
\item If furthermore there exists a collection of $d$ primitive vectors that are {\em not} part of some cone of $\Sigma$, then $Z_{\Sigma}$ has a component of codimension precisely $d$.  The set of components of $Z_{\Sigma}$ of codimension exactly $d$ is in canonical bijection with the set of (unordered) collections of $d$ primitive vectors that are {\em not} part of some cone of $\Sigma$.
\item Suppose $Z_{\Sigma}$ has codimension $d$ in ${\mathbb A}^{\Sigma(1)}$. Then the intersection of two codimension $d$ subspaces of
$Z_{\Sigma}$ in ${\mathbb A}^{\Sigma(1)}$ has codimension $d+1$ if
and only if under the above bijection the corresponding collections
have $d-1$ primitive vectors in common.
\item All intersections of codimension $d$ components of
$Z_{\Sigma}$ have codimension at least $d+2$ in
$\mathbb{A}^{\Sigma(1)}$ if and only if for every set of $d-1$
primitive vectors $\setof{\rho_1,\ldots,\rho_{d-1}}$ spanning a
cone, there exists at most one primitive vector $\rho_{d}$ such that
$\setof{\rho_1,\ldots,\rho_{d-1},\rho_d}$ does not span a cone.
\end{enumerate}
\end{prop}

\begin{proof}
Observe that $Z_{\Sigma}$ is defined by the simultaneous vanishing
of the set of cone-complement monomials:
\begin{equation*}
\begin{matrix}
x_{i_{1,1}} \cdots x_{i_{1,k_1}}=0 \\
x_{i_{2,1}} \cdots x_{i_{2,k_2}}=0 \\
\vdots \\
x_{i_{l,1}} \cdots x_{i_{l,k_l}}=0\\
\end{matrix}
\end{equation*}

By DeMorgan's Laws, the locus is the union of intersections of the
vanishing of precisely one factor $x_{j_{s,t}}$ from each monomial.
To simplify notation slightly, for the rest of the proof we use
fewer indices, reverting to the convention of the above definition:
to a cone $\sigma$ with $k$ primitive vectors $\setof{\rho_{i_1},
\ldots, \rho_{i_k}}$ the associated cone-complement monomial is $x_1
\cdots \hat{x}_{i_1} \cdots \hat{x}_{i_{k}} \cdots x_{\Sigma(1)}$.
In particular, if $\sigma$ is a cone of $\Sigma$, then the
coordinate hyperplane $x_{i_j} = 0$ is not contained in $Z_{\Sigma}$
for any $j$.

To prove the forward direction of the first claim, assume there is a
collection of $d-1$ primitive vectors $\rho_1, \ldots, \rho_{d-1}$
which are {\em not} part of a cone of $\Sigma$.  Then every monomial
contains at least one factor $x_j$ for $j \in 1, \ldots d-1$.  Thus,
by the DeMorgan's Laws argument, one of the components of
$Z_{\Sigma}$ consists of a coordinate subspace determined by a set
of equations $x_j = 0$ for $j$ drawn from ${1, \ldots, d-1}$, and
hence is codimension no more than $d-1$.

To prove the reverse direction of the first claim, observe that
since any component of $Z_{\Sigma}$ is a coordinate subspace, it is
equivalent to show that no codimension $(d-1)$ coordinate subspace
$x_{i_1} = \cdots = x_{i_{d-1}} = 0$ is contained in any component
of $Z_{\Sigma}$.  By the DeMorgan's Laws statement above, if a
component contains $x_{i_1} = \cdots = x_{i_{d-1}} = 0$, then {\em
each} cone-complement monomial must have at least one of the
$x_{i_j}$ as a factor. In other words, every cone of $\Sigma$ must
then lack at least one of the $\rho_{i_j}$, which is a contradiction
because by assumption all $(d-1)$ vectors $\rho_{i_j}$ are part of a
$(d-1)$-cone.

For the second statement, note that by the first claim all
components of $Z_{\Sigma}$ are at least codimension $d$.  Consider
primitive vectors $\rho_1, \ldots, \rho_d$ that do not form part of
a $d$-cone. By assumption each of the $d$ possible $(d-1)$-element
subsets are part of a $(d-1)$-cone; we denote by $\sigma_i$ the cone
consisting of each $\rho_j$ other than $\rho_i$.  Then $x_i$ is a
factor in the cone-complement monomial of $\sigma_i$; furthermore no
other cone of $\Sigma$ contains all of the $\rho_j$, so some $x_j$
(for $1 \leq j \leq d$) is a factor in each of the remaining
cone-complement monomials.  Consequently $x_1 = \cdots = x_d = 0$ is
contained in a component of $Z_{\Sigma}$, and because it is of the
minimal codimension $d$, it must actually be a component of
$Z_{\Sigma}$.

It follows that, under the assumptions of the first claim, given any
collection of $d$ primitive vectors $\rho_{i_1}, \ldots, \rho_{i_d}$
that are {\em not} part of a cone of $\Sigma$, there is a
canonically associated codimension $d$ subspace $x_{i_1} = \ldots =
x_{i_d} = 0$ that is a component of $Z_{\Sigma}$.  It is clear from
the construction that this canonical association is a bijection
$\beta$.

For the third claim, consider two codimension $d$ components of
$Z_{\Sigma}$, call them $L_1$, given by $x_{i_1} = \cdots =
x_{i_d}=0$, and $L_2$, given by $x_{j_1} = \cdots = x_{j_d}$.  They
intersect in a codimension $d+1$ set if and only if $d-1$ of the
$i_l$ indices agree with $d-1$ of the $j_s$ indices.  Using the
definition of the bijection, this is true if and only if the
$d$-cones $\beta^{-1}(L_1)$ and $\beta^{-1}(L_2)$ have $d-1$
primitive vectors in common.

The final claim follows because by the third claim all pairwise
intersections of such spaces are codimension at least $d+2$ if and
only if under $\beta^{-1}$ the corresponding collections of $d$
primitive vectors have strictly fewer than $d-1$ elements in common.
Equivalently, given any $(d-1)$-cone there is at most one collection
of $d$ primitive vectors of the stated type.
\end{proof}

\subsubsection*{The Cox cover as an $\aone$-covering space}
One can deduce from Proposition \ref{prop:combinatorics} that
the Cox cover of Theorem \ref{thm:cox} is a geometric Galois
$\aone$-covering space (as introduced in Definition
\ref{defn:geometriccovering}).

\begin{prop}\label{prop:galoiscover}
\label{lem:connected} Suppose $X_{\Sigma}$ is a smooth proper toric
variety over an infinite field $k$.  Both $X_{\Sigma}$ and ${\mathbb
A}^{\Sigma(1)} - Z_{\Sigma}$ are $\aone$-connected.  Thus, the
quotient morphism
$$q: {\mathbb A}^{\Sigma(1)} - Z_{\Sigma} \longrightarrow X_{\Sigma}
$$
makes ${\mathbb A}^{\Sigma(1)} - Z_{\Sigma}$ into a geometric Galois
$\aone$-covering of $X_{\Sigma}$ with group
$Pic(X_{\Sigma})^{\vee}$.
\end{prop}

\begin{proof}
If $X_{\Sigma}$ is proper and strictly positive dimensional, then $\Sigma$ must have more than one maximal cone, else it would be an affine toric variety.  A degenerate fan would correspond to a toric variety which decompose into a Cartesian product with a torus, which again would contradict being proper.  The first statement of Proposition \ref{prop:combinatorics} implies that $Z_{\Sigma}$ has codimension at least $2$ in ${\mathbb A}^{\Sigma(1)}$.  Thus, ${\mathbb A}^{\Sigma(1)} - Z_{\Sigma}$ is $\aone$-connected by Lemma \ref{lem:affinespace}.

Because $X_{\Sigma}$ is smooth and proper, it follows that
$X_{\Sigma}$ can be covered by open affine subsets isomorphic to
affine space and is thus $\aone$-connected.  Indeed, the fan
$\Sigma$ can be written as a union of maximal cones $\sigma$, each
of which is a toric variety corresponding to an affine space.  Thus
every point of $X_{\Sigma}$ is contained in an open subscheme
isomorphic to an affine space and $X_{\Sigma}$ is
$\aone$-chain-connected.

Finally, Theorem \ref{thm:cox} implies that $q: {\mathbb
A}^{\Sigma(1)} - Z_{\Sigma} \longrightarrow X_{\Sigma}$ is a
$Pic(X_{\Sigma})^{\vee}$-torsor and thus a Galois $\aone$-cover.
Thus, combining this with the previous paragraphs, we see that $q$
is a geometric Galois $\aone$-cover.
\end{proof}

\begin{rem}
Observe that the proof actually works for any smooth toric variety
which is the complement of a codimension at least $2$ subvariety of
a smooth proper toric variety.  
\end{rem}

\section{$\aone$-homotopy groups of smooth toric varieties}
\label{s:vanishing} In order to study the $\aone$-homotopy groups of
smooth proper toric varieties, we use the quotient presentation
described in the previous section.  Together with the long exact sequence in $\aone$-homotopy groups of a fibration, this will allow us to reduce our computations to the study of $\aone$-homotopy groups of complements of coordinate subspaces in affine space.

While we can not, at the moment, compute the first non-vanishing
$\aone$-homotopy group of a coordinate subspace complement in
complete generality, the combinatorial conditions described in the
Proposition \ref{prop:combinatorics} allow us to make computations
for many toric examples.  Theorem \ref{thm:abelian} provides a computation of the first non-vanishing $\aone$-homotopy group of a coordinate subspace complement in the situation where the subspace
arrangement contains no components of codimension $\leq d$, and the
codimension $d$ subspaces have pairwise intersection of dimension
$\geq d+2$.

\subsubsection*{Complements of coordinate subspaces in affine space}
Consider ${\mathbb A}^n$ with fixed coordinates $x_1,\ldots,x_n$.  Let $\setof{L_i}$, $i \in I$ be a sequence of coordinate subspaces in ${\mathbb A}^n$.  Suppose furthermore that all the $L_i$ have codimension $\geq d$.  Suppose $I' \subset I$ is a subset such that for $i \in I'$ each $L_i$ has codimension exactly $d$.  In this situation we get a morphism
$$
j: {\mathbb A}^n - \cup_{i \in I} L_i \hookrightarrow {\mathbb A}^n - \cup_{i \in I'} L_i
$$
whose complement is necessarily of codimension $\geq d+1$.

\begin{cor}
\label{cor:highercodim}
The map
$$
j_*: \pi_i^{\aone}({\mathbb A}^n - \cup_{i \in I} L_i) \longrightarrow \pi_i^{\aone}({\mathbb A}^n - \cup_{i \in I'} L_i)
$$
is an isomorphism in degrees $\leq d-1$ and a surjection in degree $d$.  Furthermore, if $I'$ is non-empty, then $\pi_{d-1}^{\aone}({\mathbb A}^n - \cup_{i \in I'} L_i)$ is non-vanishing.
\end{cor}

\begin{proof}
The first statement is an immediate corollary of Theorem \ref{thm:excision}.  If $I'$ is non-empty, we can choose any $L_i$ with $i \in I'$.  We then have an open immersion ${\mathbb A}^n - \cup_{i \in I'} L_i \hookrightarrow {\mathbb A}^n - L_i$.  This map induces a surjection on $\pi_{d-1}^{\aone}$.  Since ${\mathbb A}^n - L_i$ is $\aone$-weakly equivalent to ${\mathbb A}^d - 0$, the result follows from Theorem \ref{thm:projectivespace}.
\end{proof}

Suppose now that $L_1,\ldots,L_r$ are (distinct) coordinate subspaces of codimension (exactly) $d \geq 2$ in ${\mathbb A}^n$, then we have
${\mathbb A}^n - \cup_{i=1}^r L_i = \cap_{i=1}^r {\mathbb A}^n - L_i
$ where the intersection is taken in affine space.  Similarly,
$\cup_{i=1}^r {\mathbb A}^n - L_i = {\mathbb A}^n - \cap_{i=1}^r L_i
$ where again the intersection is taken ${\mathbb A}^n$.  Observe that, since the $L_i$ are distinct, the intersection $L_i \cap L_j$ is always a non-empty coordinate subspace of lower dimension than either $L_i$ or $L_j$. Furthermore, any such coordinate subspace complement is $\aone$-connected by Lemma \ref{lem:affinespace}.  Thus, this gives us a useful inductive procedure for computing $\aone$-homotopy groups of such complements.

\begin{lem}
\label{lem:surjectivityd}
Suppose $L_1,\ldots,L_r$ are a collection of coordinate subspaces of
codimension $d$ ($d \geq 2$) in ${\mathbb A}^n$.  There is a
surjective morphism
\[
\tau: H_{d-1}^{\aone}({\mathbb A}^n - \cup_{i=1}^r L_i) \longrightarrow {\underline{{\bf K}}^{MW}_d}^{\oplus r}.
\]
\end{lem}

\begin{proof}
We know that for each pair of subspaces $L_i,L_j$, the intersection $L_i \cap L_j$ has codimension $\geq d+1$ in ${\mathbb A}^n$ as the $L_i$ are distinct coordinate subspaces.

We will proceed by induction using the Mayer-Vietoris sequence of Proposition \ref{prop:mayervietoris}.  Consider the subspace ${\mathbb A}^n - \cup_{j=1}^m L_j$.  We can write this subspace as the intersection of ${\mathbb A}^n - \cup_{j=1}^{m-1} L_j$ and ${\mathbb A}^n - L_m$.  The union of these two subspaces is ${\mathbb A}^n - (\cup_{i=1}^{m-1} \cap L_m)$.  Consider the Mayer-Vietoris sequence for this pair of open sets.  We get
\begin{equation}
\label{eqn:mayervietorisd}
\begin{split}
\cdots \longrightarrow& {H}_{d}^{\aone}({\mathbb A}^n - ((\cup_{j=1}^{m-1}L_j) \cap L_m)) \longrightarrow {H}_{d-1}^{\aone}({\mathbb A}^n - (\cup_{j=1}^m L_j)) \\ \longrightarrow& {H}_{d-1}^{\aone}({\mathbb A}^n - \cup_{j=1}^{m-1} L_j ) \oplus {H}_{d-1}^{\aone}({\mathbb A}^n - L_m) \longrightarrow {H}_{d-1}^{\aone}({\mathbb A}^n - ((\cup_{j=1}^{m-1}L_j) \cap L_m)) \longrightarrow \cdots
\end{split}
\end{equation}
Since ${\mathbb A}^n - ((\cup_{j=1}^{m-1}L_j) \cap L_m)$ has complement of codimension $\geq d+1$ in ${\mathbb A}^n$, by excision we know its $d-1$st homotopy group vanishes.  By the $\aone$-Hurewicz theorem its $d-1$st reduced homology group vanishes as well.  Thus the last term in the diagram vanishes.  By Theorem \ref{thm:projectivespace} combined with the Hurewicz theorem, we know that ${H}_{d-1}^{\aone}({\mathbb A}^n - L_m)$ is isomorphic to $\underline{\bf{K}}^{MW}_{d}$.  Induction on $m$ gives the result.
\end{proof}

Observe that if $d \geq 3$, the $\aone$-Hurewicz theorem allows us to identify the $(d-1)$st $\aone$-homotopy and $\aone$-homology groups.  Thus, the morphism $\tau$ in the statement of Lemma \ref{lem:surjectivityd} is actually a surjection on $\aone$-homotopy groups.  The next result gives a computation of the first non-trivial $\aone$-homotopy group.

\begin{prop}
\label{prop:abelian}
Under the hypotheses of Lemma \ref{lem:surjectivityd}, if all the $L_i$ have pair-wise intersection of codimension $\geq d+2$, then the morphism $\tau$ is an isomorphism and $$\pi_{d-1}^{\aone}({\mathbb A}^n - \cup_{i=1}^r L_i) \cong {\underline{\bf{K}}^{MW}_{d}}^{\oplus r}.$$
\end{prop}

\begin{proof}
If $d \geq 3$, this follows immediately from Lemma \ref{lem:surjectivityd}.  If $d = 2$, we get an identification $\tilde{H}_1({\mathbb A}^n - \cup_{i=1}^r L_i) \cong {\underline{\bf{K}}^{MW}_{2}}^{\oplus r}$.  Thus, it suffices to prove that $\pi_1^{\aone}({\mathbb A}^n - \cup_{i=1}^r L_i)$ is abelian; the result then will follow from the $\aone$-Hurewicz theorem.  We claim that, up to homotopy, there is a surjective morphism
$$
\pi_1^{\aone}(({\mathbb A}^2 - 0)^{\times \ell}) \longrightarrow \pi_1^{\aone}({\mathbb A}^n - \cup_{i=1}^r L_i).
$$
Assuming this, observe that the first sheaf is abelian by Theorem \ref{thm:projectivespace}, so the result follows.

Choose a basis $x_1,\ldots,x_n$ of ${\mathbb A}^n$.  Observe that codimension $2$ coordinate subspaces are defined by the vanishing of pairs $x_i,x_j$ for $i,j \in \setof{1,\ldots,n}$.  By assumption, the successive intersections of the $L_i$ have codimension $\geq 4$.  A parity argument shows that there can only be an even number of codimension $2$ subspaces in ${\mathbb A}^n$ satisfying these conditions.  Thus, the number of coordinates that do not vanish for some codimension $2$ subspace is odd if $n$ is odd and even if $n$ is even.  Projection onto the coordinates that do not appear defines an $\aone$-weak equivalence to a subspace complement in ${\mathbb A}^{2n'}$ where each of the coordinates vanishes in at least one coordinate subspace.

In this last case, it is easy to construct an open embedding from a product of copies of ${\mathbb A}^2 - 0$.  By Theorem \ref{thm:excision}, this open immersion defines a surjective morphism of homotopy groups.
\end{proof}

\subsubsection*{Vanishing and non-vanishing results}
\label{s:nonvanishing}
Proposition \ref{prop:combinatorics} gives a purely combinatorial condition guaranteeing that the intersection of any pair of codimension $d$ subspaces has codimension $d + 2$.  This together with all of the previous results proved leads us to our main computation.

\begin{thm}
\label{thm:abelian}
Suppose $X_{\Sigma}$ is a toric variety corresponding to a smooth, proper fan $\Sigma$.  Let ${\mathbb A}^{\Sigma(1)}$ denote the affine space of dimension equal to the number of $1$-dimensional cones in $\Sigma$.  Let $Z_{\Sigma}$ be the irrelevant subvariety.
\begin{enumerate}
\item Let $r$ denote the number of (unordered) collections of pairs of primitive vectors $\rho_i,\rho_j$ that do not span a cone.  Suppose for each primitive vector $\rho_i$, there exists at most one index $i(j)$ such that $\rho_i,\rho_{i(j)}$ does not span a cone.  Then $\pi_1^{\aone}(X_{\Sigma})$ fits into an extension of the form:
\[
1 \longrightarrow {\underline{\bf{K}}^{MW}_2}^{\oplus r} \longrightarrow \pi_1^{\aone}(X_{\Sigma},x) \longrightarrow Pic(X_{\Sigma})^{\vee} \longrightarrow 1.
\]
\item Suppose $d \geq 3$ and every non-degenerate collection of $d-2$-primitive vectors spans a cone.  Let $r$ denote the number of (unordered) collections of $d-1$-primitive vectors that do not span a cone.  Suppose for each (unordered) collection of $d-2$ primitive vectors $\rho_1,\ldots,\rho_{d-2}$ that do span a cone, there is a unique primitive vector $\rho_{d-1}$ (distinct from the $\rho_i$) such that $\setof{\rho_1,\ldots,\rho_{d-1}}$ does not span a cone.  Furthermore, the quotient map ${\mathbb A}^{\Sigma(1)} - Z_{\Sigma} \longrightarrow X_{\Sigma}$ induces (functorial) isomorphisms
\begin{equation*}
\begin{split}
\pi_1^{\aone}(X_{\Sigma},x) &\cong Pic(X_{\Sigma})^{\vee}, \\
\pi_i^{\aone}(X_{\Sigma},x) &\cong 0 \text{ for all } 2 \leq i \leq d-2, \text{ and }\\
\pi_{d-1}^{\aone}(X_{\Sigma},x) &\cong {\underline{\bf{K}}^{MW}_d}^{\oplus r}.
\end{split}
\end{equation*}
\end{enumerate}
\end{thm}

\begin{proof}
The Cox cover is a geometric Galois $\aone$-cover by Proposition \ref{prop:galoiscover}.  Thus, Corollary \ref{cor:homotopycovering} shows that we have a corresponding long exact sequence in $\aone$-homotopy groups.  The results then follow immediately by combining Proposition \ref{prop:combinatorics}, Corollary \ref{cor:highercodim}, and Proposition \ref{prop:abelian}.
\end{proof}

\begin{rem}
Theorem 1 of \cite{Wendt} provides a complete determination of the $\aone$-fundamental group of ${\mathbb A}^{\Sigma(1)} -  Z_{\Sigma}$ for a smooth proper toric variety $X_{\Sigma}$ associated with a fan $\Sigma$.  Nevertheless, the group structure on $\pi_1^{\aone}(X_{\Sigma})$ is not completely understood.  In particular, one must specify the precise extension in question.  
\end{rem}

\subsection*{Examples}
\label{ss:examples}
In this subsection, we present some sample computations of $\aone$-homotopy groups of smooth proper toric varieties.  We will use notation following the previous sections throughout.  Thus, if $\Sigma$ is a fan, we write $X_{\Sigma}$ for the associated toric variety, $\Sigma(1)$ for the set of $1$-dimensional cones in $\Sigma$, ${\mathbb A}^{\Sigma(1)}$
for the affine space containing the Cox cover and $Z_{\Sigma}$ for
the irrelevant subvariety corresponding to $\Sigma$.

\subsubsection*{Blow-ups}
Suppose ${\sf Bl}_{Y}(X) \longrightarrow X$ is a blow-up of a smooth variety $X$ at a smooth center $Y$ having codimension $\geq 2$.  Since $Pic({\sf Bl}_Y(X))$ is necessarily isomorphic to $Pic(X) \oplus \Z$ with the generator of $\Z$ being the class of the exceptional divisor, the $\aone$-fundamental group of ${\sf Bl}_Y(X)$ is necessarily non-isomorphic to the $\aone$-fundamental group of $X$.  Thus, a blow-up of a smooth variety at a smooth subvariety having codimension $\geq 2$ can never be an $\aone$-weak equivalence.

Recall that birational morphisms of toric varieties correspond to locally-finite subdivisions of fans.  Suppose $f: X_{\Sigma'} \longrightarrow X_{\Sigma}$ is a blow-up of $X_{\Sigma}$ at a smooth toric subvariety $X_{\Delta}$ corresponding to a smooth sub-fan $\Delta \subset \Sigma$.  We will consider the induced homomorphism:
\[
f_*: \pi_1^{\aone}(X_{\Sigma'}) \longrightarrow \pi_1^{\aone}(X_{\Sigma}).
\]
Suppose $Z_{\Sigma}$ has codimension $\geq d$ in ${\mathbb A}^{\Sigma(1)}$, $d \geq 3$.  Observe that $f_*$ is surjective in this situation.  Indeed, pulling back the $\aone$-universal cover of $X_{\Sigma}$ via $f$ gives a torsor under a torus on $X_{\Sigma'}$; one can then apply the $\aone$-covering space dictionary.  The discussion in the previous paragraph shows that there is always a factor of $\gm$ in the kernel of $f_*$, and one might ask whether this is the entire kernel.  Unfortunately, we will see that this is never the case.

If $\Sigma$ is a fan, closures of orbits of codimension $d$ correspond to cones of dimension $d$.  If $\Sigma$ is a smooth fan, one can check that any cone $\sigma \subset \Sigma$ corresponds to a smooth fan as well.    If $\sigma$ is a cone in a fan $\Sigma$, let $\Star({\sigma})$ be the set of all cones $\tau \in \Sigma$ which contain $\sigma$.

\begin{construction}[Blowing up $\sigma \subset \Sigma$]
Let us describe the refinement of $\Sigma$ corresponding to the blow-up of $X_{\Sigma}$ at the (smooth) toric subvariety corresponding to $\sigma$.
\begin{enumerate}
\item Let $\rho_1,\ldots,\rho_d$ be the set of primitive generators for $\sigma$.  Set $\rho_0 = \sum_{i=1}^d \rho_i$, and let $\sigma_i$ be the cone with primitive generators $\rho_0,\rho_1,\ldots,\hat{\rho_i},\ldots,\rho_d$ (where the $\hat{\rho_i}$ indicates that $\rho_i$ has been omitted).
\item For each $\tau \in \Star(\sigma)$, decompose $\tau = \sigma + \sigma'(\tau)$, where $\sigma'(\tau) \cap \sigma = \setof{0}$ (such a decomposition exists because $\tau$ is generated by a subset of the basis for ${\rm X}^*(T)$).
\item For each $\tau \in \Star(\sigma)$, replace $\tau$ by the cones $\sigma_i + \sigma'(\tau)$.
\item Write $\Sigma'$ to be the fan with this new collection of cones.
\end{enumerate}
\end{construction}

\begin{ex}
Let $X_{\Sigma}$ be a smooth proper $n$-dimensional toric variety corresponding to a fan $\Sigma$.  Suppose we want to blow-up $X_{\Sigma}$ at a torus fixed point.  Torus fixed points correspond to cones of dimension $n$, so fix such a cone and call it $\sigma$; the star of such a cone consists of the cone itself.  If $\rho_1,\ldots,\rho_n$ are primitive vectors for $\sigma$, then we consider the vector $\rho_0 = \rho_1+ \cdots + \rho_n$ and subdivide the cone $\sigma$ into subcones of the form $\rho_0,\rho_1,\ldots,\hat{\rho_i},\ldots,\rho_n$.  Let $\Sigma'$ denote the new fan.
\end{ex}

\begin{lem}
Let $X_{\Sigma}$ be a smooth proper toric variety and let $\sigma \subset \Sigma$ be a cone.  Let $\Sigma'$ denote the refinement corresponding to the blow-up of $X_{\Sigma}$ at the subvariety corresponding to $\sigma$.  Then $\pi_1^{\aone}(X_{\Sigma'})$ is {\em never} isomorphic to $Pic(X_{\Sigma'})$.
\end{lem}

\begin{proof}
Take the vector $\rho_0$ introduced in the blow-up and any primitive vector in a $1$-dimensional cone not lying in the star of $\sigma$; this gives a pair of primitive vectors not contained in a cone and hence shows that $Z_{\Sigma}$ has a codimension $2$ component.
\end{proof}

\begin{ex}
Consider ${\mathbb P}^n$ viewed as a toric variety under $\gm^{\times n}$.  We can view ${\mathbb P}^n$ as associated with the following fan in ${\rm X}^*(T)$.  Identify the last group with $\Z^n$ and let $e_i$ denote the usual basis vectors.  Let $e_{n+1} = -(e_1+ \cdots,e_n)$, and consider the fan $\Sigma$ whose $n$-dimensional cones correspond to subsets of the form $e_1,\ldots,\hat{e_i},\ldots,e_{n+1}$.

Consider the torus fixed point corresponding to the cone $e_1,\ldots,e_n$ and let $e_{n+2} = \sum_{i=1}^n e_i$; denote this fixed-point by $x$.  Then the $n$-dimensional cones of the blow-up of ${\mathbb P}^n$ at the corresponding torus fixed point are $e_1,\ldots,\hat{e_i},\ldots,e_{n+1}$ (with $i \neq n+1$) and $e_1,\ldots,\hat{e_j},\ldots,e_n,e_{n+2}$ (with $j = 1,\ldots,n$); let $\Sigma'$ denote the corresponding fan. Observe that, so long as $n > 2$, the {\em only} pair of primitive vectors not lying in some cone is the pair $e_{n+1},e_{n+2}$.  Thus, $Z_{\Sigma'}$ has a unique codimension $2$ component and Theorem \ref{thm:abelian} shows that $\pi_1^{\aone}({\sf Bl}_x({\mathbb P}^n))$ fits into an exact sequence of the form
$$
1\longrightarrow {\underline{\bf{K}}^{MW}_2} \longrightarrow \pi_1^{\aone}({\sf Bl}_x({\mathbb P}^n)) \longrightarrow \gm^{\times 2} \longrightarrow 1.
$$
Recall also that one may identify ${\sf Bl}_x({\mathbb  P}^n)$ as a ${\mathbb P}^1$-bundle over ${\mathbb P}^{n-1}$.  We refer the reader to the next example for the case $n = 2$.   

Next, let $x'$ denote the torus fixed-point corresponding to the cone $e_2,\ldots,e_{n+1}$.  Let $e_{n+3} = \sum_{i=1}^n e_{i+1}$.  If we blow-up ${\mathbb P}^n$ at both points $x$ and $x'$, the resulting toric variety has $n$-dimensional cones given by $e_1,\ldots,\hat{e_i},\ldots,e_{n+1}$ (with $i = 2,\ldots,n$), $e_1,\ldots,\hat{e_j},\ldots,e_n,e_{n+2}$ (with $j = 1,\ldots,n$), and $e_2,\ldots,\hat{e_k},\ldots,e_{n+1},e_{n+3}$ (with $k = 2,\ldots,n+1$).  Let $\Sigma''$ be the corresponding fan.  There are now a host of pairs of primitive vectors which do not fit into some cone and one may check that some pairs corresponding to codimension $2$ subspaces in $Z_{\Sigma''}$ have intersection of codimension $3$.  Thus, while there is still a surjection $\pi_1^{\aone}({\sf Bl}_{x,x'}{\mathbb P}^n) \longrightarrow \gm^{\times 3}$, the kernel of this surjection is rather difficult to describe.
\end{ex}

\subsubsection*{Hirzebruch surfaces}
Smooth projective toric surfaces correspond to $2$-dimensional fans.  By general theory of toric surfaces, every smooth projective toric surface can be obtained from either ${\mathbb P}^2$ or a Hirzebruch surface ${\mathbb F}_a$ by successive blow ups at torus fixed points (see \cite{Fulton} \S 2.5 for more details).  For us, the surface ${\mathbb F}_a$ can be identified with the projectivization of the rank $2$ vector bundle $\O \oplus O(a)$ (for some integer $a$) over ${\mathbb P}^1$.  Morel's results describe the $\aone$-fundamental group of ${\mathbb P}^2$, so let us describe the $\aone$-fundamental group of ${\mathbb F}_a$.

\begin{ex}
The group $\pi_1^{\aone}({\mathbb F}_a)$ fits into a short exact sequence of the form
\begin{equation*}
1 \longrightarrow {\underline{\bf{K}}^{MW}_2}^{\oplus 2} \longrightarrow \pi_1^{\aone}({\mathbb F}_a) \longrightarrow \gm^{\oplus 2} \longrightarrow 1.
\end{equation*}
The group structure on this sheaf has been determined by Morel and the $\aone$-fundamental group of ${\mathbb F}_a$ depeds only on the value of $a$ mod $2$.  In other words, the extension appearing in Corollary \ref{cor:homotopycovering} can be non-trivial!
\end{ex}

\subsubsection*{Smooth complete toric varieties with small numbers of generators}
Kleinschmidt has classified all $d$-dimensional smooth proper fans with at most $d+2$-generators (see \cite{Kleinschmidt}).  Any smooth proper fan must have at least $d+1$-generators; and any such fan with exactly $d+1$-generators is necessarily the fan of projective space ${\mathbb P}^d$.

Assume now $d \geq 2$.  As before let $e_i$ denote the unit vectors in $\Z^{\oplus d}$.  Suppose we have given a pair of integers $s$ and $r$ satisfying $2 \leq s \leq d$, $r = d - s + 1$ and a collection of non-negative integers with $0 \leq a_1 \leq a_2 \leq \cdots \leq a_r$.  Consider the following collection of primitive vectors
$$
(e_1,\ldots,e_r,-\sum_{i=1}^r e_i, e_{r+1},\ldots,e_{r+s-1},\sum_{i=1}^r a_i e_i - \sum_{j=1}^{s-1} e_{r+j})
$$
Set $U = \setof{e_1,\ldots,e_r,\sum_{i=1}^r e_i}$ and $V = \setof{e_{r+1},\ldots,e_{r+s -1},\sum_{i=1}^r a_i e_i - \sum_{j=1}^{s-1} e_{r+j}}$.  We define a fan $\Sigma_d(a_1,\ldots,a_r)$ that has $U \cup V$ as its set of primitive generators and whose $d$-dimensional (maximal) cones are the positive linear spans of sets of the form $U \cup V - \setof{u,v}$ where $u \in U$ and $v \in V$.  Any smooth proper toric $d$-variety with $d+2$-generators is isomorphic to precisely one of the varieties $X_{\Sigma_d(a_1,\ldots,a_r)}$ by \cite{Kleinschmidt} Theorem 1.  Furthermore, by \cite{Kleinschmidt} Theorem 3, any such variety is isomorphic to a ${\mathbb P}^r$-bundle over ${\mathbb P}^{s-1}$.  We may thus deduce from Theorem \ref{thm:abelian} then that $\pi_1^{\aone}(X_{\Sigma_d(a_1,\ldots,a_r)})$ is isomorphic to $\gm \times \gm$ if $r \geq 2, s \geq 3$, is an extension of $\gm \times \gm$ by $\underline{{\bf K}}^{MW}_2$ if either $s = 2$ or $r = 1$ (but not both) and is an extension of $\gm \times \gm$ by ${\underline{{\bf K}}^{MW}_2}^{\oplus 2}$ if $s = 2$ and $r = 1$.

\subsubsection*{Proper, non-projective examples}
It is known that any smooth proper, non-projective toric variety must have $rk Pic(X) \geq 4$ (see \cite{KlSt}).  An example of a smooth proper, non-projective toric $3$-fold with $rk Pic(X) = 4$ was constructed by Oda (see \cite{Oda} p. 84); this example admits a morphism to projective space, though there exist smooth proper toric $3$-folds admitting no morphism to a projective variety (see e.g., \cite{PaFu}).  In all of these cases, we know that $\pi_1^{\aone}(X)$ surjects onto the torus dual to the Picard group, but we can't describe the kernel.  Nevertheless, this discussion suggests the following combinatorial question. 

\begin{question}
Does there exist a smooth proper, non-projective fan $\Sigma$ having the property that any pair of primitive vectors is contained in a cone?
\end{question}

\begin{footnotesize}
\bibliographystyle{alpha}
\bibliography{toric}
\end{footnotesize}

\end{document}